\documentclass{amsart}
\usepackage{amsmath}
\usepackage{amsthm}
\usepackage{mathtools}  % for \shortintertext
\usepackage{amssymb}
\usepackage{amsfonts}
\usepackage[utf8]{inputenc}
\usepackage{graphicx}
\usepackage{float}
\usepackage{colonequals}
\usepackage{subfigure}
\usepackage[square,numbers,sort]{natbib}
\usepackage{doi}

\usepackage{todonotes}
\usepackage{hyperref}
\usepackage{verbatim}
\usepackage{bm}

\newtheorem{theorem}{Theorem}
\newtheorem{definition}[theorem]{Definition}
\newtheorem{lemma}[theorem]{Lemma}
\newtheorem{prop}[theorem]{Proposition}

\def\N{\mathbb N}
\def\R{\mathbb R}

\newcommand{\invstproj}{p_\text{st}}

\providecommand{\abs}[1]{{\lvert#1\rvert}}
\providecommand{\norm}[1]{\lVert#1\rVert}
\DeclareMathOperator*{\argmin}{arg\,min}

\DeclareMathOperator*{\tr}{tr}

\newcommand{\Lip}{\operatorname{Lip}}

\newcommand{\vfield}[1]{\mathbf{#1}}

\title{Projection-Based Finite Elements for Nonlinear Function Spaces}

\author[Grohs]{Philipp Grohs}
\address{Philipp Grohs\\
Universität Wien\\
Fakultät für Mathematik\\
Oskar Morgenstern Platz 1\\
1090 Wien\\
Austria
}
\email{philipp.grohs@univie.ac.at}

\author[Hardering]{Hanne Hardering}
\address{Hanne Hardering\\
Technische Universität Dresden\\
Institut für Numerische Mathematik\\
Zellescher Weg 12--14\\
01069 Dresden\\
Germany}
\email{hanne.hardering@tu-dresden.de}

\author[Sander]{Oliver Sander}
\address{Oliver Sander\\
Technische Universität Dresden\\
Institut für Numerische Mathematik\\
Zellescher Weg 12--14\\
01069 Dresden\\
Germany}
\email{oliver.sander@tu-dresden.de}

\author[Sprecher]{Markus Sprecher}
\address{Markus Sprecher\\
ETH Zürich\\
Seminar für Angewandte Mathematik\\
Rämistraße~101\\
8092 Zürich\\
Switzerland}

\graphicspath{{gfx/}}

\begin{document}

\begin{abstract}
We introduce a novel type of approximation spaces for functions with values in a nonlinear manifold.
The discrete functions are constructed by piecewise polynomial interpolation in a Euclidean embedding space,
and then projecting pointwise onto the manifold. We show optimal interpolation error bounds with respect
to Lebesgue and Sobolev norms.  Additionally, we show similar bounds for the test functions, i.e.,
variations of discrete functions. Combining these results with a nonlinear Céa lemma, we prove optimal
$L^2$ and $H^1$ discretization error bounds for harmonic maps from a planar domain into a
smooth manifold.  All these error bounds are also verified numerically.
\end{abstract}

\maketitle

\noindent \emph{AMS classification:
                65N30, %  Finite elements, Rayleigh-Ritz and Galerkin methods, finite methods
                65D05  %  Interpolation
                }  \\

\noindent \emph{Keywords:} geometric finite elements, projection, nonlinear manifold, interpolation errors,
                           discretization errors, harmonic maps \\

We investigate the discrete approximation of functions from a Euclidean
domain $\Omega$ to a closed embedded submanifold $M$ of $\R^n$, $n \in \N$.
Such functions are involved in a variety of partial differential equations (PDEs), from fields like
liquid crystal physics \cite{alouges1997} and micromagnetics \cite{kohn}. In these applications,
the manifold is $M=S^2$, the unit sphere
in $\R^3$. In Cosserat-type material models \cite{sander_neff_birsan:2016,neff:2007,muench:2007}
the manifold is $M=\R^3 \times \text{SO}(3)$, where $\text{SO}(3)$ is the special orthogonal group.
Further examples are the investigation of harmonic maps into manifolds \cite{bartels_prohl:2007}, signal processing of manifold-valued signals \cite{donoho:2005}, and the denoising of manifold-valued images \cite{bacak2016second}.

We are interested in functions of Sobolev smoothness.  By this we mean functions from spaces
\begin{equation*}
 W^{k,p}(\Omega,M)
 \colonequals
 \Big\{ v \in W^{k,p}(\Omega,\R^n) \; : \; v(x) \in M\,\text{a.e.}\Big\},
\end{equation*}
where we denote by $W^{k,p}(\Omega,\R^{n})$ the standard Sobolev space for $k\in \N$
and $p\in [1,\infty]$.  Throughout the paper, $|\cdot |_{W^{l,p}}$ and $\|\cdot \|_{W^{l,p}}$ will denote the
corresponding Sobolev semi norm and full norm of $\R^n$-valued functions, respectively.

Spaces of approximating functions will be constructed by pointwise projection.
Given a finite element grid of $\Omega$, and a set of values $c_i \in M \subset \R^n$
at Lagrange points on $\Omega$, we construct nonlinear finite element functions by first interpolating
in $\R^n$ by piecewise polynomials in $\R^n$, and then projecting pointwise onto $M$. This results
in a finite-dimensional set of functions $V_{h}(\Omega,M)$ which, as it turns out, is a subspace of $W^{1,p}$
for arbitrary $p \ge 1$.
While the approach presented here is based on Lagrangian interpolation in $\R^n$, other
linear FE space can be used in principle (see~\cite{sprecher:2016} for an example).

The idea to generalize finite elements spaces by a pointwise projection operator
has already appeared several times~\cite{sprecher:2016,gawlik_leok:2017,sander:2016}.
For functions taking values in the special orthogonal group $SO(3)$, \citeauthor{gawlik_leok:2017}
have studied $C^{1}$-norms of interpolation errors~\cite{gawlik_leok:2017}.
We will extend these results to general closed submanifolds of $\R^n$, and to interpolation errors
in Sobolev norms.

To this end, let $Q_{\R^n}$ be the standard nodal interpolation operator for $n$-valued Lagrangian finite
elements, and set $Q_M \colonequals \mathcal{P} \circ Q_{\R^n}$ the interpolation operator with
a pointwise projection. For smooth manifolds, approximation qualities can be inferred from the
 linear ones of $V_{h}(\Omega,\R^{n})$, as we can switch back and forth between discrete functions
 into $M$ and into $\R^{n}$ via the definition of $Q_M$ and the identity
\begin{align*}
Q_{\R^{n}}\circ Q_{M}=Q_{\R^{n}}.
\end{align*}
An alternative proof that uses the Lipschitz continuity of the closest-point projection has
been given in \cite{sprecher:2016}.

A priori, test functions for manifold-valued settings are defined as variations of particular
manifold-valued functions. We show that test functions for functions defined by polynomial
interpolation and projection can also be constructed directly,
using Euclidean interpolation followed by a projection. We show
the same Sobolev interpolation error bounds for these discrete test functions as for the finite element functions themselves.

We then discuss finite element discretizations of PDEs with values in $M$. Prototypically,
we focus on harmonic maps from a domain $\Omega$ to $M$, which we regard as minimizers $u$ of
the Dirichlet energy in a suitable Sobolev space.
The corresponding discrete solution $u_h$ is defined as a local minimizer of the same energy
in $V_h(\Omega,M)$, which is well-defined because $V_h(\Omega,M)$ is suitably conforming.

To estimate $\|v_{h}-u\|_{W^{1,2}}$ we combine a simple nonlinear Céa lemma with the interpolation
results for $V_h(\Omega,M)$. To show optimal $L^2$ bounds we use the abstract theory of~\cite{hardering:2018},
showing that the four criteria stated there are fulfilled by projection-based finite elements.
We will also provide inverse estimates. In classical finite element theory they are used in many proofs,
e.g., in Nitsche's method of weighted norms for uniform convergence estimates \cite{ciarlet}.
In this work we will use them to justify a priori bounds on discrete minimizers of the harmonic
map energy.
Both interpolation and discretization error bounds are verified numerically in the two final chapters.

There is one alternative construction for conforming finite element spaces for manifold-valued problems,
known as geodesic finite elements~\cite{sander2012,sander2016,grohs:2013,grohs_hardering_sander:2015}.
To evaluate the relative merits of the two methods we briefly revisit their theoretical
relationship, and we repeat all numerical tests using geodesic finite elements.
We observe that while geodesic finite elements yield lower errors, projection-based finite elements
can be much faster.

\setcounter{tocdepth}{1}  % Show only section, but no subsections etc.
\tableofcontents

\section{Projection-based finite element spaces}

Let $\Omega$ be discretized by a finite union $\mathcal{G}$ of affine-equivalent, regular and quasi-uniform polyhedra $T_{h}$, such that the closures intersect in common faces.
On $\mathcal{G}$ we consider scalar-valued Lagrangian finite element spaces with the nodal basis
$(\phi_i)_{i \in I}\colon \Omega\rightarrow\R$ and associated Lagrange points
$(\xi_i)_{i\in I}\subset \Omega$.

We will define the space $V_h(\Omega,M)$ of projection-based finite elements as the image of an interpolation operator.
First we consider the canonical interpolation operator $Q_{\R^n}$ for continuous functions with values
in $\R^n$ into the space $V_{h}(\Omega,\R^{n})$ of $\R^n$-valued Lagrangian finite elements.
\begin{definition}
The interpolation operator $Q_{\R^n}\colon C(\Omega,\R^n)\rightarrow V_{h}(\Omega,\R^{n})$ corresponding
to a set of basis functions $(\phi_i)_{i \in I}\colon \Omega \to \R$ and nodes
$(\xi_i)_{i\in I}\subset \Omega$ is defined by
\begin{align*}
Q_{\R^n}v \colonequals \sum_{i\in I} v(\xi_i)\phi_{i} \quad \text{for all } v\in C(\Omega,\R^n).
\end{align*}
\end{definition}

For a manifold $M$ embedded in $\R^n$ and a function $v\in C(\Omega,M)$ the values of $Q_{\R^n}v$ will
in general not be on $M$ away from the $\xi_i$. To get $M$-valued functions we compose $Q_{\R^n}$ pointwise
with the closest-point projection
\begin{equation*}
P : \R^n \to M,
\qquad
P(q)\colonequals \argmin_{r\in M} \|r-q\|_{\R^{n}},
\end{equation*}
where $\|\cdot\|_{\R^{n}}$ denotes the Euclidean distance.
While the closest-point projection is usually not well defined for all $q\in \R^n$,
if $M$ is regular enough it is well defined in
a neighborhood $U_{\delta}\subset \R^n$ of $M$ \cite{Abatzoglou}.

This pointwise projection induces a superposition operator $\mathcal{P}$ by
\begin{equation*}
(\mathcal{P}v)(x)\colonequals P(v(x))
\end{equation*}
for all $x\in \Omega$ and $v\colon \Omega \to  \R^n$.
We then define $M$-valued interpolation by composition of $Q_{\R^n}$ and $\mathcal{P}$.

\begin{definition}
% Set
% \begin{equation*}
%  C(\Omega,M;\rho)
%    \colonequals
%  \Big\{v\in C(\Omega,M)\;:\; \forall T_{h}\in \mathcal{G} \; \exists B_{T_h,\rho} \;\text{s.t.}\; v(T_{h})\subset B_{T_h,\rho} \Big\},
% \end{equation*}
% where the $B_{T_h,\rho}$ are geodesic balls of radius $\rho$ in $M$.
 Set
 \begin{equation*}
  C(\Omega,M;\rho)
    \colonequals
  \Big\{v\in C(\Omega,M)\;:\;  \operatorname{diam}(v(T_{h}))<\rho\ \forall T_{h}\in \mathcal{G} \Big\},
 \end{equation*}
 where the $\operatorname{diam}(B)\colonequals \sup_{p,q\in B}d(p,q)$ denotes the geodesic diameter of a subset $B\subset M$.
 Provided that $\rho$ is small enough, define the interpolation operator
 \begin{align*}
   Q_{M}\colon C(\Omega,M;\rho)\to C(\Omega,M)
   \qquad \text{by} \qquad
   Q_M\colonequals \mathcal{P} \circ Q_{\R^n}.
 \end{align*}
\end{definition}

The space of projection-based finite elements is defined as the range of this interpolation operator.

\begin{definition}
 Let $\Omega \subset \R^s$, $M\subset \R^n$ an embedded submanifold, and
 $P\colon U_{\delta}\subset \R^n \to M$ the closest-point projection.
 For a given set of basis functions $(\phi_i)_{i \in I}\colon \Omega\to \R$ we define
 \begin{equation}\label{eqdef:V}
  V_h(\Omega,M)  \colonequals
    \Big\{ v_{h} : \Omega \to M \; \text{s.t.} \;  \exists\, v \in C(\Omega,M) \text{ and } v_{h}=Q_Mv \Big\}.
 \end{equation}
\end{definition}
As the operator $Q_M$ only uses the values at the Lagrange nodes $(\xi_i)_{i\in I}$, we have the
equivalent definition
 \begin{equation*}
  V_{h}(\Omega,M)  \colonequals
    \bigg\{ v_{h} : \Omega \to M,\;   \exists (c_i)_{i\in I}\subset M \text{ s.t. } v_{h}=\mathcal{P}\bigg(\sum_{i\in I} c_i \phi_i\bigg)\bigg\}.
 \end{equation*}
It has to be noted that while there exist nodal values $(c_i)_{i \in I} \subset M$, $c_i = v_h(\xi_i)$
for any function $v_h \in V_h(\Omega,M)$,
for given values $(c_i)_{i\in I}\subset M$, there exists an interpolating function $v_h\in V_h(\Omega,M)$
only if the values $c_{i}$ are close enough depending on $M$ such that $\sum_{i\in I} c_i \phi_i\subset U_{\delta}$.

\subsection{Conformity}
The question of conformity of projection-based elements, i.e., whether $V_{h}(\Omega,M) \subset W^{1,p}(\Omega,M)$ holds, can be reduced to the continuity of the superposition operator $\mathcal{P}$ on $W^{1,p}(\Omega,M)$ and of the operator $Q_{\R^{n}}$.

Denote by $P^\prime(x)[y]\in T_{P(x)}M$ the differential of the closest-point projection $P$ at $x\in U_{\delta}$ applied to $y\in T_x \R^n \simeq \R^n$.
Let $\hat v_{h}(x)=\sum_{i\in I} c_i \phi_i(x) \in \R^n$ for given coefficients $c_i$, and $v_{h}=\mathcal{P}\hat v_{h}$.
By the chain rule we have
\begin{equation*}
\frac{\partial}{\partial x_j}v_{h}(x)=P^\prime\left(\sum_{i\in I} c_i \phi_i(x)\right)\left[\sum_{i\in I} c_i \frac{\partial\phi_i}{\partial x_j}(x)\right]
\end{equation*}
for every $x\in \Omega$ such that $P$ is differentiable at $\hat v_{h}(x)$ and $\phi_i$ is differentiable
at $x\in \Omega$ for all $i\in I$.

 If we assume that $M$ is a smooth embedded submanifold, there exists a tubular neighborhood such that the closest-point projection is smooth~\cite[Prop.\,6.1.8]{lee}.
 In particular, the pointwise norm of $P'$ can be estimated in terms of the radius of curvature using explicit calculations in terms of the local parametrization of the manifold~\cite{Abatzoglou}.
 Thus, the $W^{1,p}$-conformity of $V_h(\Omega,M)$ follows directly from the chain rule and smoothness of the Lagrange basis $(\phi_i)_{i\in I} \subset W^{1,p}$,
\begin{align*}
	|v_{h}|_{W^{1,p}}\leq \|P\|_{C^{1}}\sum_{i\in I} \|c_i\|_{\R^{n}}\ |\phi_i|_{W^{1,p}(\Omega)},
\end{align*}
where $\|P\|_{C^{1}}$ denotes the operator norm of the differential $P'$.

\subsection{Relationship to geodesic finite elements}

Projection-based finite elements are closely related to the geodesic finite elements proposed in
\cite{sander2012,sander2016,grohs:2013} and analyzed in \cite{grohs_hardering_sander:2015,hardering:2015}.
Geodesic finite elements are constructed by replacing polynomial interpolation of values $(c_i)_{i\in I}$
\begin{equation*}
 x
 \mapsto
 \sum_{i\in I} c_i \phi_i(x)
 =
 \argmin_{q \in \R^n} \sum_{i\in I} \phi_i(x) \|c_i -q\|_{\R^{n}}^2
\end{equation*}
by the weighted Riemannian center of mass
\begin{equation}
\label{eq:geodesic_interpolation}
x \mapsto
\argmin_{q \in M} \sum_{i\in I} \phi_i(x) \operatorname{d}(c_i,q)^2,
\end{equation}
where $\operatorname{d}(\cdot,\cdot)$ is the geodesic distance on $M$.
Unlike the construction by pointwise projection, \eqref{eq:geodesic_interpolation} is completely intrinsic,
and does not rely on an embedding space.
Well-posedness of this definition under suitable conditions on the $c_i$ is shown in \cite{sander2016,hardering:2015}.

As observed independently by \cite{sprecher:2016} and \cite{gawlik_leok:2017}, we recover
the projection-based interpolation if we replace the geodesic distance in \eqref{eq:geodesic_interpolation} by the Euclidean distance
of the embedding space
\begin{align*}
\argmin_{q \in M} \sum_{i\in I} \phi_i(x) \norm{c_i - q}_{\R^n}^2
&=\argmin_{q \in M} \bigg(\norm{q}^2_{\R^n}-2\bigg\langle q,\sum_{i\in I} \phi_i(x)c_i\bigg\rangle\bigg) \\
&=\argmin_{q \in M} \bigg\lVert q-\sum_{i\in I} \phi_i(x)c_i\bigg\rVert^{2}_{\R^n}\\
&= P\bigg(\sum_{i\in I} \phi_i(x)c_i\bigg).
%&= Q_Mv(x).
\end{align*}
This does not mean that projection-based finite elements are equal to geodesic finite elements
for embedded manifolds. In general, even if the metric on $M$ is obtained by an isometric embedding
into Euclidean space, the distance $\operatorname{d}(\cdot,\cdot)$ is not the Euclidean distance in the surrounding space.
Instead, projection-based interpolation can be interpreted as geodesic finite elements for a general
metric space $(M,d)$ with a non-intrinsic metric. As far as we know, no general existence theory
and error estimates exist for this abstract setting.

\subsection{Preservation of isometries}

If an isometry $T\colon M\rightarrow M$ commutes with the projection-based interpolation operator $Q_M$
then the finite element space $V_{h}(\Omega,M)$ defined in \eqref{eqdef:V} is equivariant under this isometry.
In mechanics, this leads to the desirable property that discretizations of objective problems are again objective.
Unfortunately, for projection-based finite elements this commutativity only holds under special
circumstances.

\begin{definition}
An isometry $T\colon M\to M$ (w.r.t.\ the geodesic distance) is called extendable if there exists an isometry $\tilde{T}\colon \R^n\to \R^n$ with $\tilde{T}(p)=T(p)$ for all $p\in M$.
\end{definition}
Examples for extendable isometries $T$ are orthogonal transformations for the sphere and multiplication with special orthogonal matrices for $SO(n)$.

\begin{theorem}
Let $M\subset \R^n$ be a Riemannian submanifold, $P : \R^n \to M$ the closest-point projection,
$T\colon M\rightarrow M$ an extendable isometry and $(\phi_i)_{i\in I}$ a partition of unity. Then $T$ commutes with $Q_M$.
\end{theorem}
\begin{proof}
Let $\tilde{T}$ be an extension of $T$ to $\R^n$.
As an isometry maps closest distances to closest distances, $\tilde{T}$ commutes with $P$. By the Mazur--Ulam theorem \cite{Ulam} there exists a linear map $A\colon \R^n \rightarrow \R^n$ with $\tilde{T}(q)=A(q)+\tilde{T}(0)$ for all $q\in \R^n$, so that $\tilde{T}$ obviously commutes with $Q_{\R^n}$.
\end{proof}
One can tell from this proof that only very few isometries $T : M \to M$ are extendable.
Indeed, in order to be extendable, $T$ needs to be the restriction of a rigid body motion of $\R^n$.
In contrast to this rather strong restriction, the geodesic interpolation rule~\eqref{eq:geodesic_interpolation}
is equivariant under any isometry of $M$ by construction.

\subsection{Discrete test functions and vector field interpolation}\label{sec::testfct}

The test function space for a function $u:\Omega \to M$ consists of vector fields along $u$ that correspond to intrinsic variations within the class of functions considered. For $u\in H^1(\Omega,M)$, we call the space of test functions $H^1(\Omega, u^{-1}TM)$.
If we consider $M\subset \R^{n}$ as an embedded submanifold, then $H^1(\Omega,u^{-1}TM)$ can be canonically identified with a subset of $H^1 (\Omega,\R^{n})$.

We construct discrete test functions in the same manner, i.e., $\vfield{v}_{h}$ is a discrete test function for $u_{h}\in V_h(\Omega,M)$
if there exists a variation
$\gamma:\Omega\times(-\epsilon,\epsilon)\to M$ such that $\gamma(\cdot,t)\in V_h(\Omega,M)$ for all $t\in (-\epsilon,\epsilon)$,  $\gamma(\cdot,0)=u_{h}$, and $\frac{d}{dt}\gamma(\cdot ,0)=\vfield{v}_{h}$ \cite{hardering:2015,sander:2016}.
Writing this definition using the coefficients $(c_i)_{i \in I} \subset M$ that constitute $u_h$,
the set of all discrete test functions over the discrete function $u_h$ can be defined as
\begin{multline*}
W_{h}(\Omega, u_h^{-1}TM)
\colonequals \bigg\{\vfield{v}_{h} \in L^2(\Omega,\R^{n})\;:\; \exists (c_{i})_{i\in I}\subset C((-\epsilon,\epsilon),M)\ \textrm{s.t.}\\ \vfield{v}_{h}=\frac{d}{dt}\bigg|_{t=0}\mathcal{P}\bigg(\sum_{i\in I}c_{i}(t)\phi_{i}\bigg)\bigg\}.
\end{multline*}
Similar to the discrete functions themselves, discrete test functions can be constructed by
polynomial interpolation followed by pointwise projection, as by chain rule we have for any $\vfield{v}_h \in W_h$ and
$x\in \Omega$
\begin{align*}
\vfield{v}_{h}(x)&= P^{\prime}\bigg(\sum_{i\in I}c_{i}(0)\phi_{i}(x)\bigg)\bigg(\sum_{i\in I}c'_{i}(0)\phi_{i}(x)\bigg),
\end{align*}
and $P^\prime(y):T_y\R^{n}\to T_{P(y)}M$, the differential of the closest-point projection $P$,
is again a projection.
\begin{prop}
	Let $P: U_{\delta}\subset \R^{n}\to M$ be the closest-point projection onto a closed embedded $C^2$-submanifold
 $M\subset \R^n$.
 For any $y \in U_{\delta}$, the differential $P^\prime(y):T_y\R^{n}\to T_{P(y)}M$
 is the orthogonal projection onto the tangent space $T_{P(y)}M$,
 with the canonical interpretation of $T_{P(y)}M$ as a subspace of $\R^n$.
\end{prop}

\begin{proof}
	Let $y\in U_{\delta}$.
	We need to show that for all $\xi\in T_{y}\R^{n}$ and $\omega \in T_{P(y)}M \subset \R^{n}$
	\begin{align*}
	 \langle P^{\prime}(y)(\xi) - \xi, \omega\rangle= 0
	\end{align*}
	holds. To see this, we consider the curve $c:(-\epsilon,\epsilon)\to M$ defined by $c(t)=P(y+t\xi)$,  and a vector field $w(t)$ along $c$ with $w(0)=\omega$.
	As $P$ is defined by minimization, the first variation yields at any $t\in (-\epsilon,\epsilon)$
	\begin{align*}
	\langle P(y+t\xi) - (y+t\xi), w(t)\rangle =0.
	\end{align*}
	Differentiating this with respect to $t$ yields
	\begin{align*}
	0&= \frac{d}{dt}\bigg|_{t=0}\langle P(y+t\xi) - (y+t\xi), w(t)\rangle\\
	&=\langle P^{\prime}(y)(\xi) - \xi, \omega\rangle + \langle P(y) - y, w'(0)\rangle\\
	&= \langle P^{\prime}(y)(\xi) - \xi, \omega\rangle. \qedhere
	\end{align*}
\end{proof}

Thus, the computation of the value at $x \in \Omega$ of a test functions along a discrete function $u_{h}\in V_h(\Omega,M)$ corresponds to
first interpolating given tangent vectors $c'_{i}\in T_{u_{h}(\xi_{i})}M\subset \R^{n}$ in $\R^{n}$,
and then projecting the resulting piecewise polynomial function pointwise orthogonally to $T_{u_{h}(x)}M$.
Alternatively, by linearity we can first project the $c'_{i}$ orthogonally to $T_{u_{h}(x)}M$, and then interpolate the result in the vector space $T_{u_{h}(x)}M$.

In particular, we can define for $u_{h}=\mathcal{P}\left(\sum_{i\in I}u_{i}\phi_{i}\right)\in V_{h}(\Omega,M)$
an interpolation operator $Q_{u_{h}^{-1}TM}:C(\Omega,u_{h}^{-1}TM)\to W_{h}(\Omega,u_{h}^{-1}TM)$ by
\begin{align*}
Q_{u_{h}^{-1}TM}\vfield{v}=\mathcal{P}_{u_{h}^{-1}TM} \circ Q_{\R^{n}}(\vfield{v}).
\end{align*}
Given some test function $\vfield{v}$ along a continuous function $u$, i.e. $(u,\vfield{v})\in C(\Omega,TM)$,
we can first interpolate $u$ and then $\vfield{v}$, as the interpolation of $\vfield{v}$ depends only on the
values at the Lagrange nodes $(\xi_i)_{i \in I}$, where $u$ and $u_{h}$ agree.

\section{Interpolation error estimates}\label{sec::errest}

In this chapter we will estimate the interpolation errors of $Q_{\R^n}$ and $Q_{M}$ in terms of the mesh width $h$.
We also estimate the error of the test vector field interpolation operator $Q_{u_h^{-1}TM}$.

\subsection{Properties of Euclidean interpolation}

Proving interpolation error bounds for $Q_M$ uses several standard results for interpolation
in Euclidean spaces.  We repeat some of them here for convenience.

Define the usual grid dependent Sobolev norms
\begin{align*}
\|v\|_{W^{l,p}(\mathcal{G})}\colonequals \left(\sum_{T\subset \mathcal{G}}\|v\|_{W^{l,p}(T)}^{p}\right)^{\frac{1}{p}}
\end{align*}
for functions $v\in C(\Omega)$ such that $v_{|T}\in W^{l,p}(T)$ for all $T\subset \mathcal{G}$.
For the rest of this paper, this norm is meant whenever we speak of the $W^{l,p}$-norm of a
discrete function, unless explicitly stated otherwise.
As we assume shape regularity of the mesh, one can use the Sobolev embedding theorem and elementwise scaling to the reference element to prove that if $lp>s$, $Q_{\R^n}$ is continuous with respect to the grid-dependent $W^{l,p}$-norm, i.e., there exists $C>0$ such that we have
\begin{align}\label{equ::loc_lip}
\left\|Q_{\R^{n}} v\right\|_{W^{l,p}(\mathcal{G})}
\leq C \left\|v\right\|_{W^{l,p}(\Omega)} \text{ for all } v\in W^{l,p}(\Omega).
\end{align}
Note that by the Sobolev embedding theorem, $Q_{\R^{n}} v$ is well-defined for all $v\in W^{l,p}$ with $lp>s$.
Under these assumptions, we have the following approximation error estimate for $Q_{\R^{n}}$ \cite{ciarlet, braess}.
\begin{theorem}\label{thm::lin}
Let $\Omega\subset \R^s$ be a bounded Lipschitz domain, $\mathcal{G}$ a shape-regular, affine-equivalent mesh on $\Omega$, $l\in \N$, $p\in [1,\infty]$
and $(\phi_{i})_{i\in I}$ Lagrangian nodal basis functions for polynomial order $r\geq l-1$.
Then on each element $T$ for any $m\in \N$ with $m\geq l$ and $\min(m,r+1)>\frac{s}{p}$ we have
\begin{align*}
\left|v-Q_{\R^n}v\right|_{W^{l,p}(T,\R^{n})}\leq C\; h_{T}^{\min(m,r+1)-l}|v|_{W^{\min(m,r+1),p}(T,\R^{n})} \quad \forall v\in W^{m,p}(T,\R^n),
\end{align*}
with the constant independent of $v$ and $h_{T}=\mathrm{diam}(T)$.
\end{theorem}
We will also need the following inverse inequalities.
\begin{theorem}\label{thm:inverLin}
	Consider a shape-regular, affine-equivalent, quasi-uniform mesh $\mathcal{G}$ and two pairs $(k,p)$ and $(m, q)$ with $1\leq k\leq m \leq \infty$ and $p, q\in  [1,\infty]$ such that the space of polynomials up to degree $r$ on $T$ is a subspace of $W^{k,p}(T)\subset W^{m,q}(T)$ for each mesh element $T\subset  \mathcal{G}$.
	Then for all discrete functions $Q_{\R^{n}}v$ of polynomial order $r$
	\begin{align*}
	|Q_{\R^{n}}v|_{W^{m,q}(T)}\leq C\; h^{-(m-k) - s\max\big\{0,\frac{1}{p}-\frac{1}{q}\big\}}|Q_{\R^{n}}v|_{W^{k,p}(T)},
	\end{align*}
	where the constant depends on the quasi-uniformity and regularity parameters of the mesh, but not on $h$.
\end{theorem}
This result here is not as general
as it could be. Inverse inequalities with weaker requirements of the mesh appear, e.g., in \cite{dahmen2004inverse} (no quasi-uniformity), and \cite{graham2005finite} (no shape regularity).
We expect that these generalizations can help to extend the following results on $M$-valued interpolation as well.

\subsection{$M$-valued interpolation}
We now turn to error bounds for the $M$-valued interpolation operator $Q_M \colonequals \mathcal{P} \circ Q_{\R^n}$.
Given  $v\in W^{m,p}(\Omega,M)$, we estimate the error $\|Q_{M}v-v\|_{W^{l,p}(\mathcal{G})}$, $l< \min(m,r+1)$, by observing that $ Q_{\R^{n}}(Q_{M}v) = Q_{\R^{n}}v$, and using the triangle inequality
\begin{align*}
\|Q_{M}v-v\|_{W^{l,p}}
&\leq \|Q_{M}v-Q_{\R^{n}}(Q_{M}v)\|_{W^{l,p}}+ \|Q_{\R^{n}}v-v\|_{W^{l,p}},
\end{align*}
(again in the grid-dependent norm).
Denoting $\hat{m} \colonequals \min(m,r+1)$,
both terms on the right can be bounded using Theorem~\ref{thm::lin}, and we obtain
\begin{align*}
\norm{Q_{M}v-v}_{W^{l,p}}
&\leq C\;  h^{\hat{m}-l}(|Q_M v|_{W^{\hat{m},p}} + |v|_{W^{\hat{m},p}}).
\end{align*}
It remains to show estimates of $|Q_{M}v|_{W^{\hat{m},p}}$
in terms of Sobolev norms of $v$.
Unlike in the Euclidean case the Sobolev semi-norm $|v|_{W^{\hat{m},p}}$ is not by itself sufficient
to bound $|Q_{M}v|_{W^{\hat{m},p}}$, because lower-order derivatives appear by the chain rule.
The proper quantity is the homogeneous norm $|\cdot|_{W^{\hat{m},p}} +|\cdot|_{W^{1,\hat{m}p}}^{\hat{m}}$, known, e.g., from \cite{superposVector}.
It replaces the unwieldy smoothness descriptor used in corresponding results for geodesic
finite elements~\cite{grohs_hardering_sander:2015}.

\begin{prop}\label{prop:QMest}
 Let $\hat m\geq 1$, $p\in[1,\infty]$, and $M$ such that the closest-point projection $P$
 is in $W^{\hat{m}+1,\infty}$ in some $\delta$-neighborhood $U_{\delta} \subset \R^{n}$ of $M$.
 Let $Q_{\R^{n}}v$ be a discrete function from a Lipschitz domain $T\subset \R^{s}$ into $\R^{n}$
 defined by interpolation of values $(c_{i})_{i\in I}\subset M$. Suppose the $(c_{i})_{i\in I}$ are
 contained in a geodesic ball $B_\rho$ of radius $\rho$, where $\rho$ is small enough such that $Q_{\R^{n}}v(T)\subset U_{\delta}$.
	Then
	\begin{multline}
	\label{eq:bound_on_qmv}
	|Q_{M}v|_{W^{\hat{m},p}}\leq |Q_{\R^{n}}v|_{W^{\hat{m},p}}\\
	 +C\rho L(P)\left(|Q_{\R^{n}}v|_{W^{\hat{m},p}}+ \max\left\{|Q_{\R^{n}}v|_{W^{1,\hat{m}p}}^{\hat{m}},|Q_{\R^{n}}v|_{W^{1,\hat{m}p}} \right\}\right),
	\end{multline}
	where $L(P)$ is a constant that depends on the $W^{\hat{m}+1,\infty}$-norm of $P$.
\end{prop}

\begin{proof}
Let $\vec{a}\in \N^s$ be a multi-index with $|\vec{a}|_1=\hat{m}$. By the chain rule, the derivative $D^{\vec{a}}Q_{M}v= D^{\vec{a}}(\mathcal{P}(Q_{\R^{n}}v))$ can be written almost everywhere as a sum of terms of the form
\begin{align*}
P^{(k)}(Q_{\R^{n}}v(x))\left[D^{\vec{a_1}} Q_{\R^{n}}v(x),...,D^{\vec{a_k}}Q_{\R^{n}}v(x)\right],
\end{align*}
where $1\leq k\leq \hat{m}$, $\vec{a_1},\dots,\vec{a_k}\in \N^s \backslash \{(0,\dots,0) \}$ and $\vec{a_1}+\dots+\vec{a_k}=\vec{a}$. An expansion of $P$ around $M$ yields
\begin{align}\label{eq::P1est}
	\|\mathcal {P}'(Q_{\R^{n}}v)\|_{L^{\infty}}\leq 1 + \Lip(P')d(Q_{\R^{n}}v,M)\leq 1 + C \Lip(P')\rho,
\end{align}
where $d(Q_{\R^{n}}v,M)\colonequals\sup_{x\in T}\inf_{z\in M}\|Q_{\R^{n}}v(x)-z\|_{\R^{n}}$, and $\Lip(P')$ denotes the Lipschitz constant of the map $P':\R^{2n}\to TM$.
For $k\geq 2$
\begin{align}\label{eq::Pkest}
	\|\mathcal {P}^{(k)}(Q_{\R^{n}}v)\|_{L^{\infty}}\leq L(P)d(Q_{\R^{n}}v,M)\leq C L(P)\rho,
\end{align}
where the constant $L$ depends on the $W^{k+1,\infty}$-norm of $P$.
Further, we have
\begin{multline*}
\Big\|P^{(k)}(Q_{\R^{n}}v(x))\left[D^{\vec{a_1}} Q_{\R^{n}}v(x),...,D^{\vec{a_k}}Q_{\R^{n}}v(x)\right] \Big\|_{L^{p}} \\
\leq \big\|\mathcal {P}^{(k)}(Q_{\R^{n}}v)\big\|_{L^{\infty}}\prod_{i=1}^{k}\big\|D^{\vec{a_i}} Q_{\R^{n}}v \big\|_{L^{\frac{\hat{m}p}{|\vec{a_i}|}}}.
\end{multline*}
For $k=1$, this yields
\begin{align*}
\|P'(Q_{\R^{n}}v(x))\left[D^{\vec{a}} Q_{\R^{n}}v(x)\right]\|_{L^{p}}
&\leq (1 + C\Lip(P')\rho)\|Q_{\R^{n}}v\|_{W^{\hat{m},p}}.
\end{align*}
For $k\geq 2$, we have by the Gagliardo--Nirenberg--Sobolev and Young's inequalities,
\begin{align*}
\|D^{\vec{a_i}} Q_{\R^{n}}v\|_{L^{\frac{\hat{m}p}{|\vec{a_i}|}}}
&\leq C |Q_{\R^{n}}v|_{W^{\hat{m},p}}^{\frac{|\vec{a_i}|-1}{\hat{m}-1}}|Q_{\R^{n}}v|_{W^{1,\hat{m}p}}^{1-\frac{|\vec{a_i}|-1}{\hat{m}-1}} + C|Q_{\R^{n}}v|_{W^{1,\hat{m}p}}\\
&\leq |Q_{\R^{n}}v|_{W^{\hat{m},p}}^{\frac{|\vec{a_i}|}{\hat{m}}} + C |Q_{\R^{n}}v|_{W^{1,\hat{m}p}}^{|\vec{a_i}|} + C|Q_{\R^{n}}v|_{W^{1,\hat{m}p}}.
\end{align*}
Combining all of this yields
\begin{align*}
	\|D^{\vec{a}}Q_{M}v\|_{L^{p}}
	&\leq \sum  \left\|P^{(k)}(Q_{\R^{n}}v(x))\left[D^{\vec{a_1}} Q_{\R^{n}}v(x),...,D^{\vec{a_k}}Q_{\R^{n}}v(x)\right]\right\|_{L^{p}}\\
	&\leq  (1 + C\Lip(P')\rho)	\|D^{\vec{a}}Q_{\R^{n}}v\|_{L^{p}}\\
	&\qquad +  C L(P)\rho \left(|Q_{\R^{n}}v|_{W^{\hat{m},p}}+ \max\left\{|Q_{\R^{n}}v|_{W^{1,\hat{m}p}}^{\hat{m}},|Q_{\R^{n}}v|_{W^{1,\hat{m}p}} \right\}\right).\qedhere
\end{align*}
\end{proof}
 If $m\geq r+1$, then the highest-order derivatives $D^{(\hat m)}Q_{\R^{n}}v$ vanish.
 In that case, \eqref{eq:bound_on_qmv} reduces to
	\begin{align*}
	|Q_{M}v|_{W^{\hat{m},p}}\leq C\;\rho L(P) \max\left\{|Q_{\R^{n}}v|_{W^{1,\hat{m}p}}^{\hat{m}},|Q_{\R^{n}}v|_{W^{1,\hat{m}p}} \right\}
	\end{align*}
 for the mesh-dependent norm.

We can now state the main theorem.

\begin{theorem}\label{thm::main_appr}
Consider the same setting as in Theorem~\ref{thm::lin}.
Let $M\subset \R^n$ be an embedded submanifold, such that the closest-point projection $P\colon U\subset \R^n\rightarrow M$ is in $W^{\hat{m},\infty}$, where $\hat{m}\colonequals \min(m,r+1)>\frac{s}{p}$.
Then there exists $\alpha>0$, depending on $m,p$, and $s$ such that for all $v\in W^{m,p}(\Omega,M)$
and $0\leq l\leq \hat{m}$
\begin{equation}
\label{eq:main_interpolation_bound}
\abs{v-Q_Mv}_{W^{l,p}}
\leq
C\;  h^{\hat{m}-l}\\
\Big[|v|_{W^{\hat{m},p}}+ \;L(P)h^{\alpha} \|v\|_{W^{\hat{m},p}}\left(\|v\|_{W^{\hat{m},p}}+ |v|_{W^{1,\hat{m}p}}^{\hat{m}}\right) \Big],
\end{equation}
with $L(P)$ as in Proposition~\ref{prop:QMest}, and the constant $C$ depending on the constant in that proposition, the one in Theorem~\ref{thm::lin}, as well as $m$, $p$, and $\Omega$.
\end{theorem}

\begin{proof}
For $l=0$, we have, using~\eqref{eq::P1est}
and Theorem~\ref{thm::lin},
\begin{align*}
\|Q_{M}v-v\|_{L^{p}}&\leq \int_{0}^{1}\|\mathcal{P}'(v+t(Q_{\R^{n}}v-v))(Q_{\R^{n}}v-v)\|_{L^{p}}\;dt\\
&\leq (1+C\;\Lip(P')\rho) \|Q_{\R^{n}}v-v\|_{L^{p}}\\
&\leq C\: (1+\;\Lip(P')\rho) h^{ \hat{m}}|v|_{W^{\hat{m},p}}.
\end{align*}
For $l\geq 1$, we use $Q_{\R^n}\circ Q_M = Q_{\R^n}$, Theorem~\ref{thm::lin}, and Proposition~\ref{prop:QMest} to estimate
\begin{align*}
\|Q_{M}v-v\|_{W^{l,p}}
&\leq \|Q_{M}v-Q_{\R^{n}}(Q_{M}v)\|_{W^{l,p}}+ \|Q_{\R^{n}}v-v\|_{W^{l,p}}\\
&\leq C\;  h^{\hat{m}-l}\left(|v|_{W^{\hat{m},p}} + |Q_{M}v|_{W^{\hat{m},p}}\right)\\
&\leq C\;  h^{\hat{m}-l}\big(|v|_{W^{\hat{m},p}} + |Q_{\R^{n}}v|_{W^{\hat{m},p}} \\
&\quad  +C\rho L(P)\left(|Q_{\R^{n}}v|_{W^{\hat{m},p}}+ \max\left\{|Q_{\R^{n}}v|_{W^{1,\hat{m}p}}^{\hat{m}},|Q_{\R^{n}}v|_{W^{1,\hat{m}p}} \right\}\right)\big).
\end{align*}
If $\{\xi_{i,T_{h}}\}_{i}$ denote the Lagrangian interpolation nodes in an element $T_{h}$, we have by the Sobolev embedding theorem for some $\alpha>0$
\begin{align*}
	\rho\leq \max_{T_{h}}\max_{i,j}d(v(\xi_{i,T_{h}}), v(\xi_{j,T_{h}}))\leq C\;\|v\|_{C^{0,\alpha}}|\xi_{i,T_{h}}-\xi_{j,T_{h}}|^{\alpha}\leq C\;\|v\|_{W^{\hat{m},p}}h^{\alpha}.
\end{align*}
By Theorem~\ref{thm::lin}, we can estimate all arising semi-norms of $Q_{\R^{n}}v$ by corresponding semi-norms of $v$.
Further, we have by the Sobolev embedding theorem $|v|_{W^{1,\hat{m}p}}\leq C\; \|v\|_{W^{\hat{m},p}}$. 
This yields the assertion.
\end{proof}
Note that the constants of our estimates are all independent of $M$. The only dependence
on $M$ is the factor $L(P)$. However, since $L(P)$ appears in the error bound~\eqref{eq:main_interpolation_bound}
only multiplied with $h^\alpha$, $\alpha > 0$, it becomes irrelevant for $h \to 0$.
The bounds are therefore optimal in terms of the mesh width.
Extra terms compared to the linear result can be controlled by the closeness parameter $\rho$ of the interpolation nodes, and thus for continuous functions by the mesh width parameter~$h$.

We have seen that, due to the chain rule, estimates on $Q_{M}$ obtained from the ones on $Q_{\R^{n}}$ are always
with respect to the homogeneous Sobolev seminorms of the type $|\cdot|_{W^{m,q}} +|\cdot|_{W^{1,mq}}^{m}$.
As the term $|\cdot|_{W^{1,mq}}^{m}$
does not scale correctly, we cannot expect general inverse estimates in the style of Theorem~\ref{thm:inverLin} for $Q_{M}$. An exception is the special case $m=1$.

\begin{theorem}\label{thm:inverse}
	Let the assumptions of Theorem~\ref{thm:inverLin} be fulfilled with $k=m=1$ and $p>s$. Then for all projected finite element functions $Q_{M}v$ we have
	\begin{align*}
		|Q_{M}v|_{W^{1,q}(T)}\leq C\; h^{- s\max\{0,\frac{1}{p}-\frac{1}{q}\}}|Q_{M}v|_{W^{1,p}(T)}.
	\end{align*}
\end{theorem}

\begin{proof}
	The estimate follows directly from Proposition~\ref{prop:QMest} and Theorem~\ref{thm:inverLin}. We need the condition $p>s$ in order to apply Theorem~\ref{thm::lin} to estimate
	\begin{align*}
		|Q_{\R^{n}}v|_{W^{1,p}}\leq |Q_{M}v|_{W^{1,p}}+ |Q_{\R^{n}}Q_{M}v-Q_{M}v|_{W^{1,p}}\leq C\; |Q_{M}v|_{W^{1,p}}.&&&
		\qedhere
	\end{align*}
\end{proof}

\subsection{$TM$-valued interpolation}
In Section~\ref{sec::testfct}, we have defined interpolation of a vector field $\vfield{v}\in C(\Omega,u_{h}^{-1}TM)$ along a discrete function $u_{h}$ by $(Q_{u_{h}^{-1}TM}\vfield{v})(x) \colonequals P'(u_{h}(x))(Q_{\R^{n}}\vfield{v}(x))$. This definition is very similar to that of $M$-valued interpolation, with the difference that the pointwise projection $P'(Q_{\R^{n}}u_{h}(x))$ is even linear.
This linearity makes proving optimal interpolation error bounds for vector fields along given
discrete functions much easier than proving the error bounds for the discrete functions themselves.

\begin{theorem}\label{thm::TM_appr}
	Let $\Omega\subset \R^s$ be a bounded Lipschitz domain, $mp>s$, $u_{h}\in V_{h}(\Omega,M)$ such that $\|u_{h}\|_{W^{m,p}}$ is bounded independently of $h$, $\vfield{v}\in W^{m,p}\cap C(\Omega,u_{h}^{-1}TM)$, and $0\leq l< \min(m,r+1)$. Let the assumptions of Theorem~\ref{thm::lin} be satisfied. Assume further that $P$ is in $W^{l+1,\infty}$ on $U_\delta$.
	Then there exist constants $C > 0$ and $\alpha > 0$ such that
	\begin{multline*}
		|Q_{u_{h}^{-1}TM}\vfield{v}-\vfield{v}|_{W^{l,p}}\\
		\leq C\;	h^{\min(m,r+1)-l} |\vfield{v}|_{W^{\min(m,r+1),p}}\Big[1 + C h^{\alpha}(1+\|u_{h}\|_{W^{m,p}} + \|u_{h}\|_{W^{1,mp}}^{m})\Big].
	\end{multline*}
\end{theorem}

\begin{proof}
	As $\vfield{v}(x)=P'(u_{h}(x))(v(x))$, we can estimate for $l=0$, using~\eqref{eq::P1est}
		\begin{align*}
		\|Q_{u_{h}^{-1}TM}\vfield{v}-\vfield{v}\|_{L^{p}}
		&=
		\left(\int_{\Omega}\big| P'(u_{h}(x))(Q_{\R^{n}}\vfield{v}(x)-\vfield{v}(x)) \big|^{p}\,dx\right)^{\frac{1}{p}}\\
		&\leq \|P'(u_{h}(x)\|_{L^{\infty}}\|Q_{\R^{n}}\vfield{v}-\vfield{v}\|_{L^{p}}\\
		&\leq (1+C\;h^{\alpha})\|Q_{\R^{n}}\vfield{v}-\vfield{v}\|_{L^{p}}.
	\end{align*}
	For $l>0$, we have using the chain rule
	\begin{multline*}
		|Q_{u_{h}^{-1}TM}\vfield{v}-\vfield{v}|_{W^{l,p}}
		\leq \|P'(u_{h}(x)\|_{L^{\infty}}\|Q_{\R^{n}}\vfield{v}-\vfield{v}\|_{W^{l,p}}\\
		+C\sum_{k=1}^{l}\|\mathcal{P}^{(k+1)}(u_{h})\|_{L^{\infty}}\left(\|u_{h}\|_{W^{k,\frac{s}{k}}}+\|u_{h}\|_{W^{1,s}}^{k}+\|u_{h}\|_{W^{1,\frac{s}{k}}}\right)\|Q_{\R^{n}}\vfield{v}-\vfield{v}\|_{W^{l-k,\frac{ps}{s-kp}}}.
	\end{multline*}
	Using~\eqref{eq::P1est} and~\eqref{eq::Pkest}, we obtain
		\begin{multline*}
		|Q_{u_{h}^{-1}TM}\vfield{v}-\vfield{v}|_{W^{l,p}}
		\leq \|Q_{\R^{n}}\vfield{v}-\vfield{v}\|_{W^{l,p}}\\
		+C\;h^{\alpha}\sum_{k=1}^{l}\left(\|u_{h}\|_{W^{k,\frac{s}{k}}}+\|u_{h}\|_{W^{1,s}}^{k}+\|u_{h}\|_{W^{1,\frac{s}{k}}}\right)\|Q_{\R^{n}}\vfield{v}-\vfield{v}\|_{W^{l-k,\frac{ps}{s-kp}}}.
	\end{multline*}
	Application of Theorem~\ref{thm::lin} concludes the proof.
\end{proof}

\section{Discretization error estimates for harmonic maps}\label{disc_err_bnd}

We use the interpolation results of the previous section to show optimal discretization error bounds for
projection-based finite element approximations of harmonic maps.  For an open domain $\Omega\subset \R^{s}$ with
piecewise $C^{1}$-boundary, a smooth Riemannian manifold $(M,g)$, and a smooth map $v:\Omega\to M$,
we define the harmonic energy by
\begin{align}
\label{eq:harmEnergy}
 \mathcal{J}(v)\colonequals \frac{1}{2}\int_{\Omega} |dv|_{g}^{2}\;dx.
\end{align}
For a review on harmonic maps between Riemannian manifolds we refer to \cite{Wood}. In the following we will always assume
that $\mathcal{J}$ has a unique local minimizer within $H_{\varphi}^{1}(\Omega, M)$,
where $H_{\varphi}^{1}(\Omega, M)$ is the set of $H^1(\Omega,M)$-functions $v$ in the same homotopy class
as the given function $\varphi$, and with $v = \varphi$ on $\partial \Omega$.

We construct discrete harmonic maps $u_h$ by minimizing $\mathcal{J}$ in the projection-based
finite element space $V_h(\Omega,M)$.
Generalizing the approach for the Euclidean case, we prove $H^1$ error bounds by combining a
nonlinear Céa lemma with
an interpolation error bound.  We then use the non-Euclidean Aubin--Nitsche trick of~\cite{hardering:2018}
to obtain bounds on the $L^2$-error.

\subsection{Ellipticity}

We start with a definition of ellipticity for manifold-valued functions.
Unlike~\cite{grohs_hardering_sander:2015}, we define ellipticity with respect to an extrinsic error
measure.  The definition is locally equivalent to the intrinsic definition of~\cite{grohs_hardering_sander:2015},
but is easier to use in the case of embedded manifolds.
In the following, $B_\epsilon^\infty(v)$ denotes the closed $L^\infty$-ball of radius $\epsilon$
centered in $v$.
\begin{definition}\label{def::ellipt}
Let $M\subset \R^n$ be an embedded submanifold and $\Omega\subset \R^s$ a domain. A functional $\mathcal{J}\colon H\subset H^1(\Omega,M)\rightarrow \R$ is called
\emph{$H^1$-elliptic} around $u\in H$ if there exist $\lambda,\Lambda,\epsilon>0$ such that for all $v\in B_{\epsilon}^\infty(u)$
we have
$$
\lambda|v-u|^2_{H^1}\leq  \mathcal{J}(v)-\mathcal{J}(u)\leq \Lambda|v-u|^2_{H^1}.
$$
\end{definition}
Having ellipticity it is straightforward to prove a nonlinear C\'ea lemma.
\begin{lemma}\label{lem::cea}
Let $M\subset \R^n$ be a Riemannian submanifold of $\R^n$, $\Omega\subset \R^s$,
and $\mathcal{J}\colon H\subset H^1(\Omega,M)~\rightarrow~\R$ a functional with a minimizer $u\in H$
that is unique in a closed ball $B^\infty_\epsilon(u)$.
Assume that $\mathcal{J}$ is elliptic around $u$. Let $V\cap B^\infty_{\epsilon}(u)\subset H$ be a nonempty subset and
$$
v \colonequals \argmin_{w\in V\cap B^\infty_{\epsilon}(u)} \mathcal{J}(w).
$$
Then
\begin{equation*}
\left|v-u\right|_{H^1} \le \sqrt{\frac{\Lambda}{\lambda}} \inf_{w\in V\cap B^\infty_{\epsilon}(u)} |w-u|_{H^1}.
\end{equation*}
\end{lemma}
\begin{proof}
By the ellipticity we have for any $w\in V$
$$
\lambda \abs{v-u}^2_{H^1}\leq \mathcal{J}(v)-\mathcal{J}(u)\leq \mathcal{J}(w)-\mathcal{J}(u)\le \Lambda \abs{w-u}^2_{H^1}.
$$
Taking the square root yields the desired result.
\end{proof}

In the following $\mathcal{J}$ will always denote the Dirichlet energy~\eqref{eq:harmEnergy}.
If $M$ is isometrically immersed in $\R^{n}$, it is well-known \cite{eels_lemaire:1978, Wood} that the Euler--Lagrange equation for critical points of $\mathcal{J}$ is
\begin{align}\label{eq:harmEulerLagrange}
\Delta u + \tr A(du,du)=0,
\end{align}
where $A$ denotes the second fundamental form of $M$. Thus, $u\in H^{2}(\Omega,M)$ is a critical point of $\mathcal{J}$
if
\begin{align}
\label{eq:laplace_is_normal}
\Delta u(x)\in T_{u(x)}M^{\perp}
\end{align}
almost everywhere.
Written in local coordinates, \eqref{eq:harmEulerLagrange} is a semilinear second-order elliptic system of partial differential equations. We show that it is also elliptic in the sense of Definition~\ref{def::ellipt}.
\begin{prop}\label{prop::harm_is_ellipt}
Let $M$ be a Riemannian submanifold of $\R^n$ such that the closest-point projection
$P\colon \R^n \supset U \to M$ is in $C^{2}$, $\Omega\subset \R^s$ with Poincaré constant $C_P$,
and $u\in W^{1,q}(\Omega,M)$, $q>\max\{2,s\}$, a critical point of the harmonic energy $\mathcal{J}$.
Let $\kappa$ denote the largest principal curvature of $M$.
Then if $\|d u\|_{L^{q}}<(\kappa C_P)^{-1} $, the functional $\mathcal{J}$ is elliptic around $u$.
\end{prop}
\begin{proof}
For $v\in H_\varphi^{1}(\Omega,M)$ we have
\begin{align}
\label{eq:harmonic_energy_difference}
\mathcal{J}(v)-\mathcal{J}(u)
&=\frac{1}{2}|v-u|^2_{H^1}-\langle v-u,\Delta u\rangle_{L^2}.
\end{align}
Let $\gamma(x,\cdot):[0,1]\to M$ be a smooth family of curves connecting $u$ and $v$ pointwise.
Then $\langle\dot \gamma(\cdot, 0),\Delta u\rangle=0$ by~\eqref{eq:laplace_is_normal}, and
\begin{align*}
\langle v-u,\Delta u \rangle(x)
&= \int_{0}^{1}(1-t) \langle \ddot \gamma (x,t),\Delta u(x)\rangle \;dt\\
&= -\int_{0}^{1}(1-t) \big\langle \ddot \gamma (x,t),\tr A (du(x),du(x))\big\rangle \;dt
\end{align*}
for almost every $x$ in $\Omega$.
Suppose that $\gamma$ is even a geodesic homotopy; then $\ddot \gamma= A(\dot\gamma, \dot \gamma )N$, where $N$ is the outer normal to $M$.
Thus we obtain, using the Poincar\'e inequality, the estimate
\begin{align*}
|\langle v-u,\Delta u \rangle_{L^{2}}|&\leq \frac{\kappa^{2} C_{P}^{2} }{2} \|d u\|_{L^{q}}^{2} |u-v|_{H^{1}}^{2}<\frac{1}{2}|u-v|_{H^{1}}^{2}.
\end{align*}
Plugging this back into~\eqref{eq:harmonic_energy_difference} yields the assertion.
\end{proof}

\subsection{Discretization error estimates in $H^1$}

We will now combine the Céa Lemma and the approximation properties of the space $V_h(\Omega,M)$.

\begin{theorem}\label{T:H1Discerr}
Let the assumptions of Proposition~\ref{prop::harm_is_ellipt} be fulfilled, and let the local
minimizer $u$ of the harmonic energy be in $H^m(\Omega,M)$ with $m>\frac{s}{2}$.
Further assume that the assumptions of Theorem~\ref{thm::main_appr} are met for this $m$ and $p=2$.
Additionally suppose that $\|d u\|_{L^{q}}<(\kappa C_P)^{-1} $, where $\kappa$ denotes the largest
principal curvature of $M$, and $C_P$ is the Poincar\'e constant of $\Omega$.
Let $\varphi$ be such that $ V_{h}\cap H_{\varphi}(\Omega,M)$ is not empty, and set
\begin{equation*}
V_{h;K}
\colonequals
\Big\{w\in V_{h}\cap H_{\varphi}(\Omega,M)\;:\; \|dw\|_{L^{q}}\leq K \Big\},
\end{equation*}
with $K$ large enough and $h$ small enough such that $Q_{M}u\in V_{h;K}\cap B_\epsilon^\infty(u)$, where $\epsilon$ is small enough such that $\mathcal{J}$ is elliptic in an $\epsilon$-neighborhood of $u$.
Set
\begin{align}\label{eq:discProb}
u_{h}\colonequals \argmin_{w\in V_{h;K}\cap B_\epsilon^\infty(u)} \mathcal{J}(w).
\end{align}
Then for $h$ small enough and $\hat{m}\colonequals\min(m,r+1)$ we have
\begin{align}\label{eq:H1err}
\left|u_{h}-u\right|_{H^1}\leq C\sqrt{\frac{\Lambda}{\lambda}}  h^{\hat{m}-1}|u|_{H^{\hat{m}}}.
\end{align}

If the grid is shape-regular, affine-equivalent, and quasi-uniform, the map $u_{h}$ is indeed a local minimizer of $\mathcal{J}$ in $ V_{h}\cap H_{\varphi}(\Omega,M)$, if we additionally assume that $q<\frac{2s}{s-2\min(m-1,r)}$ in case $2\min(m-1,r)<s$.
\end{theorem}

\begin{proof}
	By restriction to $V_{h;K}\cap B_\epsilon^\infty(u)$ and the choice of $K$ and $\epsilon$ we can apply Lemma~\ref{lem::cea} and Theorem~\ref{thm::main_appr} to obtain for $h$ small enough that
\begin{align*}
\left|u_{h}-u\right|_{H^1}
 %&\leq \sqrt{\frac{\Lambda}{\lambda}} \left|Q_Mu-u\right|_{H^1}\\
& \leq C \sqrt{\frac{\Lambda}{\lambda}}  h^{\hat{m}-1}\Big[|u|_{H^{\hat{m}}}+C\;L(P)h^{\alpha} \|u\|_{H^{\hat{m}}}(\|u\|_{H^{\hat{m}}} + |v|_{W^{1,2\hat{m}}}^{\hat{m}}) \Big]\\
& \leq C \sqrt{\frac{\Lambda}{\lambda}}  h^{\hat{m}-1} |u|_{H^{\hat{m}}}.
\end{align*}
To show that $u_h$ is a local minimizer in $V_h \cap H_\varphi$,
let $k\colonequals \frac{2s}{s-2(\hat{m}-1)}$ if $2(\hat{m}-1)<s$, and arbitrarily large otherwise.
By assumption $\max\{s,2\}<q<k$. Thus, we can set $\mu\colonequals(2^{-1} - q^{-1})(q^{-1}-k^{-1})^{-1}$, $\beta\colonequals (1+\mu)^{-1}(\hat{m}-1-s(2^{-1} - q^{-1}))>0$, and use $L^{p}$-interpolation, the $H^1$ bound \eqref{eq:H1err},
the Sobolev Embedding Theorem, and Theorem~\ref{thm:inverse} to estimate
\begin{align*}
|u_{h}-u|_{W^{1,q}}&\leq h^{-(\hat{m}-1)+\delta}|u-u_{h}|_{H^{1}} + h^{\mu^{-1}(\hat{m}-1 -\delta)} |u-u_{h}|_{W^{1,k}}\\
&\leq Ch^{\delta}|u|_{W^{\hat{m},p}} + h^{\mu^{-1}(\hat{m}-1 -\delta)} \left(|u|_{W^{1,k}}+ |u_{h}|_{W^{1,k}}\right)\\
&\leq Ch^{\delta}(\|u\|_{W^{\hat{m},p}}  + K).
\end{align*}
As $\|u_h-u\|_{L^{\infty}}\leq C|u_{h}-u|_{W^{1,q}}\leq Ch^{\delta}$, we can choose $h$ small enough such that $\|du_{h}\|_{L^{q}}< K$ and $\|u_{h}-u\|_{L^{\infty}}< \epsilon$, so that $u_{h}$ is indeed a local minimizer in $V_{h}\cap H_{\varphi}(\Omega,M)$.
\end{proof}

\subsection{Discretization error estimates in $L^2$}

To obtain optimal $L^2$-discretization error estimates we apply \cite[Thm.\,2.13]{hardering:2018}, which is a generalization of the Aubin--Nitsche Lemma~\cite[Thm.\,3.2.5]{ciarlet}.
With slightly adapted notation, it states the following:
\begin{theorem}
\label{thm:l2_generic_bound}
	Let $m>\frac{s}{2}$ and assume that $u\in H_{\varphi}^{m}(\Omega,M)$ is a minimizer of an elliptic, semi-linear  energy $E:H_{\varphi}^{m}(\Omega,M)\to \R$, that has an $H^{2}$-regular dual problem, i.e., for all $\vfield{g}\in L^{2}(\Omega,u^{-1}TM)$ there exists a solution $\vfield{w}\in H^{2}(\Omega,u^{-1}TM)$ to
	\begin{align*}
	\vfield{w}\in H^{1}_{0}(\Omega,u^{-1}TM)\qquad \delta^{2}E(u)(\vfield{w},\vfield{v})=-(\vfield{g},\vfield{v})\qquad \forall \vfield{v}\in H^{1}_{0}(\Omega,u^{-1}TM)
	\end{align*}
	with
	\begin{align*}
	\|\vfield{w}\|_{W^{2,2}}\leq C\;\|\vfield{g}\|_{L^{2}}.
	\end{align*}
	For a given shape-regular, affine-equivalent, quasi-uniform grid $\mathcal{G}$, let $S_{h}\subset H_{\varphi}(\Omega,M)$ be a discrete approximation space, such that
	for all $v\in H_{\varphi}\cap C(\Omega,M;\rho)$ with $v|_{T_{h}}\in H^{m}(T_{h},M)$ for all $T_{h}$, there exists an approximating map $v_{I}\in S_{h}$ with
	\begin{align}\label{eq:cond1b}
	|v_{I}|_{H^{m}} + |v_{I}|_{W^{1,2m}}^{m}\leq C\;	\left(|v|_{H^{m}} + |v|_{W^{1,2m}}^{m}\right)
	\end{align}
	that fulfills the estimate
	\begin{align}\label{eq:cond1a}
	\|v-v_{I}\|_{L^{2}} + h\;	|v-v_{I}|_{W^{1,2}} \leq
	C(v) \;h^{l}
	\end{align}
	for $l\geq 2$.
	Assume further that each discrete map $v_{h}\in S_{h}\cap C(\Omega,M;\rho)$ fulfills inverse estimates of the form
	\begin{align} \label{eq:inverse2p1qScaled}
	|v_{h}|_{W^{1,p}}&\leq C\;h^{-d\max\left\{0,\frac{1}{q}-\frac{1}{p}\right\}}\;|v_{h}|_{W^{1,q}}
	\end{align}
	for $q>\max\{2,s\}$.
	Finally, assume that for each vector field $\vfield{v}\in C(\Omega,v_{h}^{-1}TM)$ along a discrete function $v_{h}\in S_{h}$ that is in  $W^{2,2}(T_{h},v_{h}^{-1}TM)$ for each element $T_{h}$, there exist a variation $\vfield{v}_{I}$ of maps in $S_{h}$ such that
	\begin{align}\label{eq:vecInterpolCond}
	\| \vfield{v} - \vfield{v}_{I}\|_{W^{1,2}} &\leq C(v_{h},\vfield{v}) h.
	\end{align}
	Then if the discrete minimizer
	\begin{align*}
	u_{h}\colonequals \argmin_{v_{h}\in S_{h}} E(v_{h})
	\end{align*}
	fulfills the a priori error estimate
	\begin{align}\label{eq:H1apriori}
	|u-u_{h}|_{W^{1,2}}\leq C(u) h^{l-1},
	\end{align}
	and on each element $T_{h}$ the estimate
	\begin{align} \label{eq:2bBound}
	\|u_{h}\|_{W^{2,p_{2}}(T_{h})}^{p_{2}} + \|u_{h}\|_{W^{1,2p_{2}}(T_{h})}^{2p_{2}}\leq K_{2}^{p_{2}}
	\end{align}
	for a constant $K_{2}$ and $p_{2}=\max\{2,\frac{s}{2}\}$ for $s\neq 4$ (and $p_{2}>2$ for $s=4$), we get
	\begin{align*}
	\|u-u_{h}\|_{L^{2}}\leq C(u)\,h^{l}.
	\end{align*}
\end{theorem}
The main difference to the Euclidean Aubin--Nitsche lemma is that in order to compare test vector fields along
$u$ and $u_{h}$, one needs to be transported into the space of the other along a suitable connecting curve.
This transportation needs to preserve $H^2$-norms of vector fields.
In \cite{hardering:2018} this preservation is proven for the case that the functions $u$ and $u_h$
have bounded grid-dependent $H^2$-norm. This leads to the additional assumption~\eqref{eq:2bBound}.

For technical reasons, we restrict ourselves to the practically relevant case $s<4$. Other dimensions my be dealt with similarly as discussed in \cite{hardering:2018}.
\begin{theorem}
\label{thm:L2_discretization_error_bound}
		Consider the setting of Theorem~\ref{T:H1Discerr}
		with $m\geq 2$, $s<4$ and $q\geq 4$, including the assumptions for  $u_{h}$ to be a local minimizer in $ V_{h}\cap H_{\varphi}(\Omega,M)$.
		Then
		\begin{align*}
		\|u-u_{h}\|_{L^2}\leq C\;h^{\min(m,r+1)}.
		\end{align*}
\end{theorem}

\begin{proof}
  We prove the assertion by verifying the assumptions of Theorem~\ref{thm:l2_generic_bound}.
	First note that the harmonic map energy is indeed elliptic (Proposition~\ref{prop::harm_is_ellipt}),
	semi-linear, and has an $H^2$-regular dual problem~\cite{hardering:2015}.

	The approximation error estimates for maps~\eqref{eq:cond1b} and~\eqref{eq:cond1a} are provided by Proposition~\ref{prop:QMest} and Theorem~\ref{thm::main_appr}, respectively.
	The inverse estimate~\eqref{eq:inverse2p1qScaled} is given in Theorem~\ref{thm:inverse}.
	The interpolation error bound~\eqref{eq:vecInterpolCond} for vector fields follows from Theorem~\ref{thm::TM_appr}.
	The $H^1$ a priori bound~\eqref{eq:H1apriori} is proven in Theorem~\ref{T:H1Discerr}.

	Thus, all that is left to show is~\eqref{eq:2bBound}, i.e., for a solution $u_h$ of \eqref{eq:discProb} we need to show that there exists
	a constant $K_2$ such that the grid-dependent homogeneous $W^{2,2}\cap W^{1,4}$-norm of $u_{h}$ is bounded, i.e.,
\begin{align}
	\label{eq:h1_inverse}
		\int_{\Omega}|du_{h}|^{4}\;dx + \sum_{T_{h}\in \mathcal{G}}\int_{T_{h}}|\nabla^{2}u_{h}|^{2}\;dx\leq K_{2}^{2}.
\end{align}
The boundedness of the first integral even without the elementwise partition follows because the assumptions
$q\geq 4$ and $u_{h}\in V_{h;K}\cap B_\epsilon^\infty(u)$ imply $\|du_{h}\|_{L^{4}}\leq C\,K$.

Set $\hat u_{h}=Q_{\R^{n}}u_{h}$, and $\hat u_{I}= Q_{\R^{n}}u$. As the continuity of $Q_{M}$ implies that $\|\nabla^{2}u_{I}\|_{L^2(\mathcal{G})}$ is bounded, by the triangle inequality it is enough to obtain an estimate on $\|\nabla^{2}u_{h}-\nabla^{2}u_{I}\|_{L^{2}(\mathcal{G})}$.
	By the chain rule we have
	\begin{align*}
		\|\nabla^{2}u_{h}-\nabla^{2}u_{I}\|_{L^{2}}
		&\leq  \|\mathcal{P}'(\hat{u}_{h})(\nabla^{2}\hat{u}_{h})-\mathcal{P}'(\hat{u}_{I})(\nabla^{2}\hat{u}_{I})\|_{L^2}\\
		&\qquad  + \|\mathcal{P}''(\hat{u}_{h})(d\hat{u}_{h},d\hat{u}_{h})
		-\mathcal{P}''(\hat{u}_{I})(d\hat{u}_{I},d\hat{u}_{I})\|_{L^2}\\
		&\leq  \|\mathcal{P}'\|_{L^{\infty}}\|\nabla^{2}\hat u_{h}-\nabla^{2}\hat u_{I}\|_{L^2}\\
		&\qquad + \Lip(\mathcal{P}')\|\nabla^{2}\hat u_{I}\|_{L^{2}}\|\hat u_{h} - \hat v_{h}\|_{L^{\infty}}\\
		&\qquad  + \|\mathcal{P}''\|_{L^{\infty}}\left(\|d\hat{u}_{h}\|_{L^{4}}^{2}+\|d \hat{u}_{I}\|_{L^{4}}^{2}\right).
	\end{align*}
	By~\eqref{equ::loc_lip}, the $W^{1,4}$-bounds on $u_{h}$ and $u_{I}$ transfer to $\hat u_{h}$ and $\hat u_{I}$, and $\|\nabla^{2}\hat u_{I}\|_{L^{2}}$ is bounded. Further, we choose an exponent $a = a(s)$ such that $W^{1,2}\hookrightarrow L^{a}$ which exists by the Sobolev embedding theorem
	 Then by Theorem~\ref{thm:inverLin} we have
	\begin{align*}
		\|\nabla^{2}\hat u_{h}-\nabla^{2}\hat u_{I}\|_{L^2} + \|\hat u_{h}-\hat u_{I}\|_{L^\infty}
		&\leq C(h^{-1} + h^{-\frac{s}{a}})\|\hat u_{h}-\hat u_{I}\|_{W^{1,2}}.
	\end{align*}
	Thus, by~\eqref{equ::loc_lip} and Theorem~\ref{T:H1Discerr}, we have
		\begin{align*}
	\|\nabla^{2}\hat u_{h}-\nabla^{2}\hat u_{I}\|_{L^2} + \|\hat u_{h}-\hat u_{I}\|_{L^\infty}
	&\leq C\; h^{\min\{m-1,r\}-\max\{1,\frac{s}{a}\}}.
	\end{align*}
	By the Sobolev embedding theorem, for $s\leq 4$ we can even choose the $a$ such that $a \ge s$.
	Then $\max\{1,\frac{s}{a}\}=1$, and from $m\geq 2$ it follows that $\min\{m-1,r\}\geq 1$.
	Thus, $\|\nabla^{2}u_{h}\|_{L^{2}}$ can be bounded independently of $h$, and we obtain~\eqref{eq:h1_inverse}.
\end{proof}

\section{Numerical interpolation error tests}

We now show numerically that the optimal interpolation error orders predicted by
Theorem~\ref{thm::main_appr} can really
be observed in practice.  We test this for maps into the unit sphere $S^2$ and into $\text{SO}(3)$.
All algorithms are implemented in C++ using the \textsc{Dune}
libraries~\cite{bastian_et_al_II:2008}.

\subsection{Maps into the unit sphere}
\label{sec:interpolation_errors_sphere}

\begin{figure}
 \begin{center}
    \includegraphics[width=0.3\textwidth]{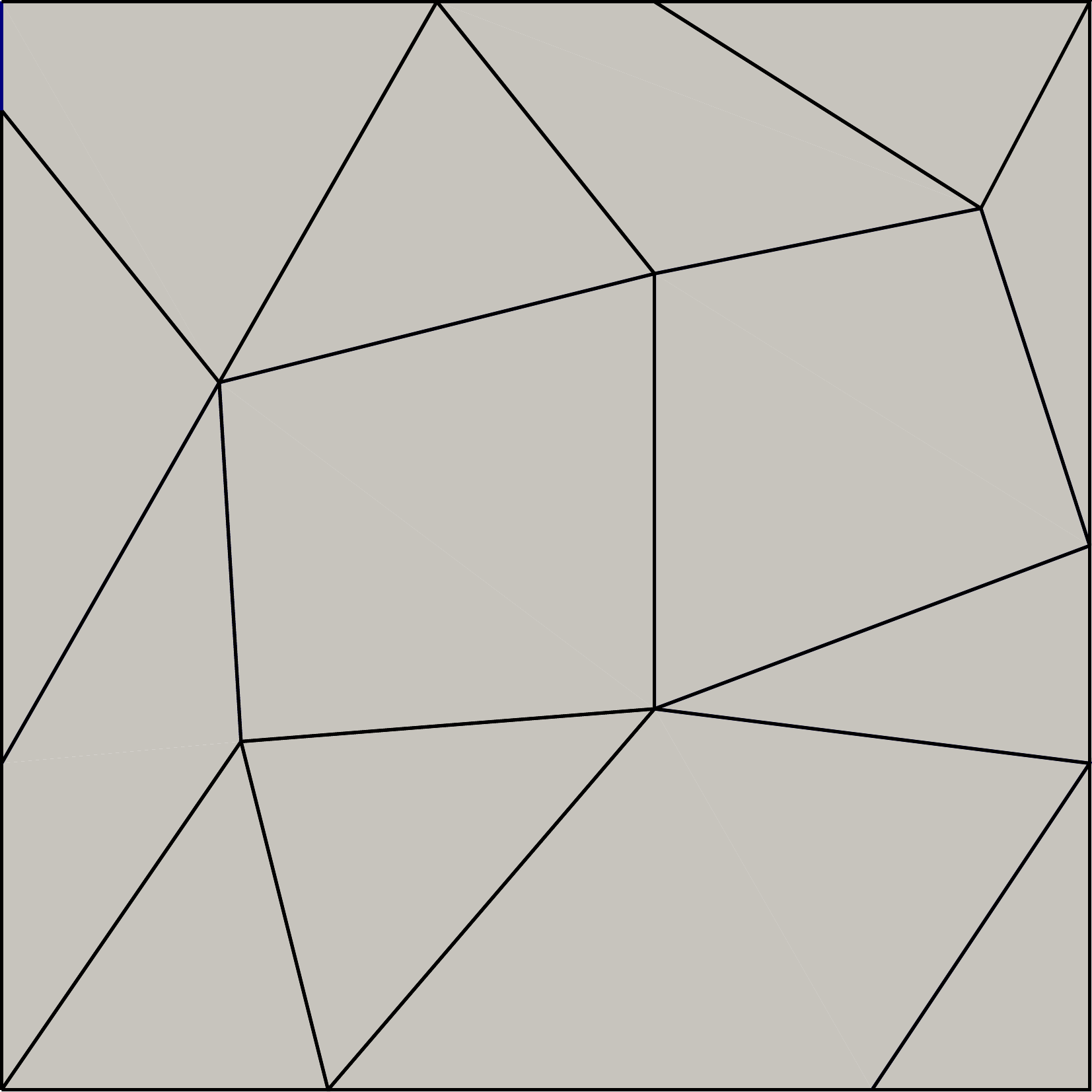}
    \hspace{0.08\textwidth}
    % Use the 'trim' command to fix the vertical alignment
    \includegraphics[width=0.6\textwidth, trim=0 -50 0 0]{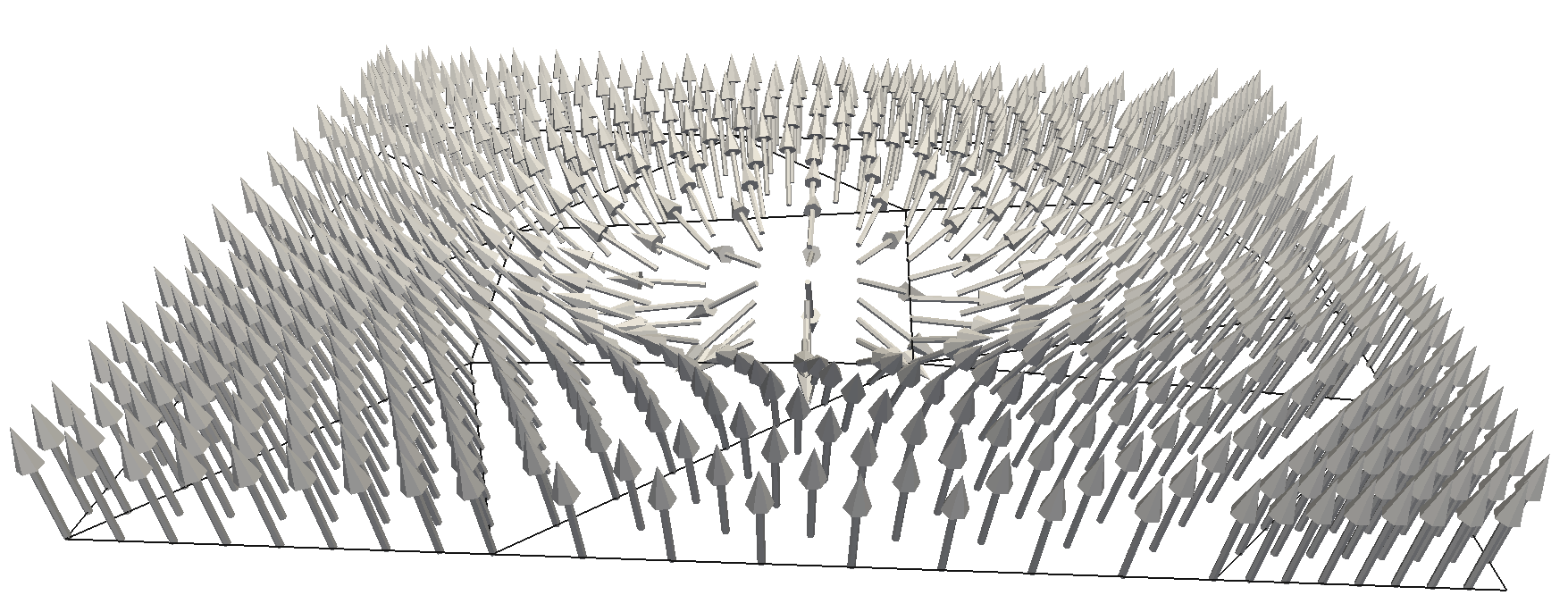}
 \end{center}
\caption{Interpolating the inverse stereographic projection $\invstproj : \R^2 \to S^2$.
   Left: coarsest grid. Right: function values}
\label{fig:inverse_stereographic_projection}
\end{figure}

Our first example measures the $L^2$ and $H^1$ interpolation errors for maps into the unit sphere $S^2 \subset \R^3$.
As the domain we use the square $\Omega = (-5,5)^2$, and we measure the error of interpolating the
inverse stereographic projection
\begin{equation}
\label{eq:inverse_stereographic_projection}
 \invstproj : \R^2 \to S^2,
 \qquad \qquad
 \invstproj(x)
 \colonequals
 \bigg( \frac{2x_0}{\abs{x}^2+1}, \frac{2x_1}{\abs{x}^2+1}, \frac{\abs{x}^2-1}{\abs{x}^2+1} \bigg)^T,
\end{equation}
restricted to $\Omega$.
This function is in $C^\infty$, and we can therefore hope for optimal interpolation error orders.

We discretize the domain with the grid shown in Figure~\ref{fig:inverse_stereographic_projection}.
Observe that it combines triangles and non-affine quadrilateral elements.
It is therefore slightly beyond the assumptions of Theorem~\ref{thm::main_appr}.
We create a sequence of grids by
refining the initial grid uniformly up to six times. On each grid we compute $Q_M \invstproj$
using projection-based finite elements of orders $p=1, 2, 3$, and we measure
the error $\norm{Q_M \invstproj - \invstproj}$ both in the $L^2(\Omega,\R^3)$-norm
and the $H^1(\Omega,\R^3)$-seminorm.
Sixth-order Gaussian quadrature rules are used for the integrals, but
note that since projection-based finite element functions are not piecewise polynomials in $\R^3$,
a small additional error due to numerical quadrature remains.

\begin{figure}
 \begin{center}
  \subfigure[Projection-based finite elements]{
   \includegraphics[width=0.49\textwidth]{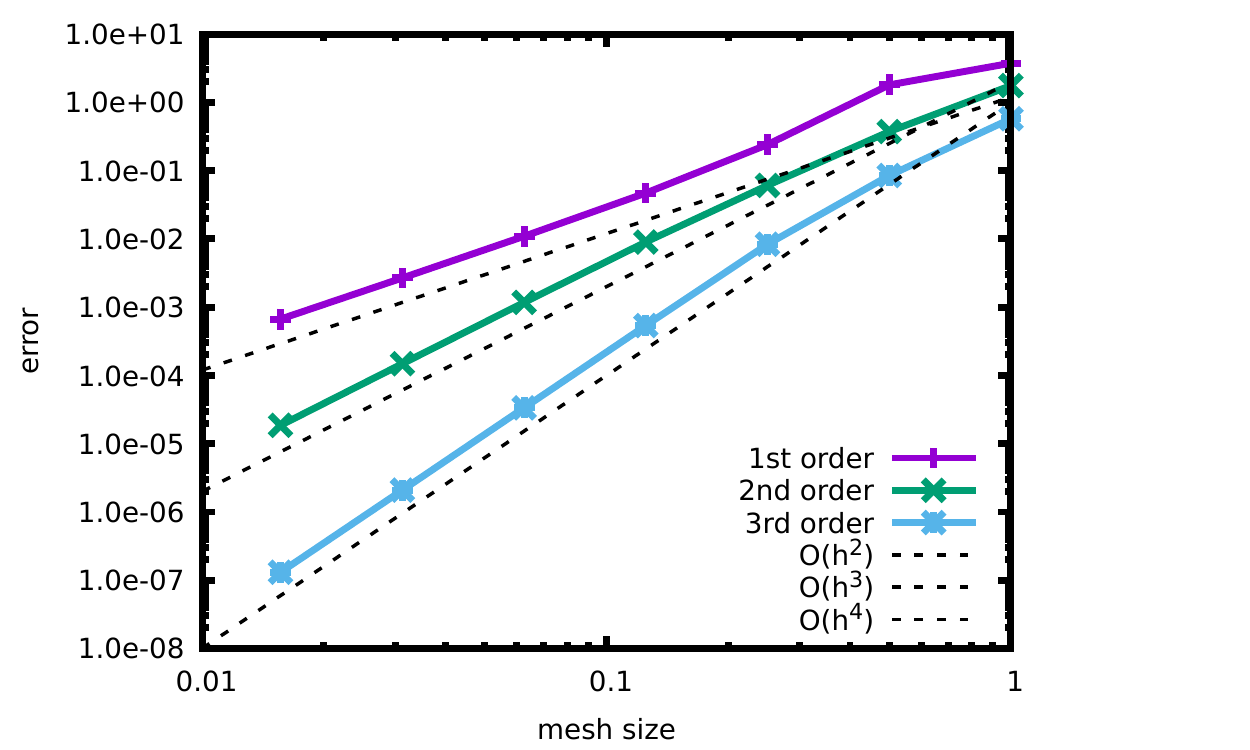}
   \includegraphics[width=0.49\textwidth]{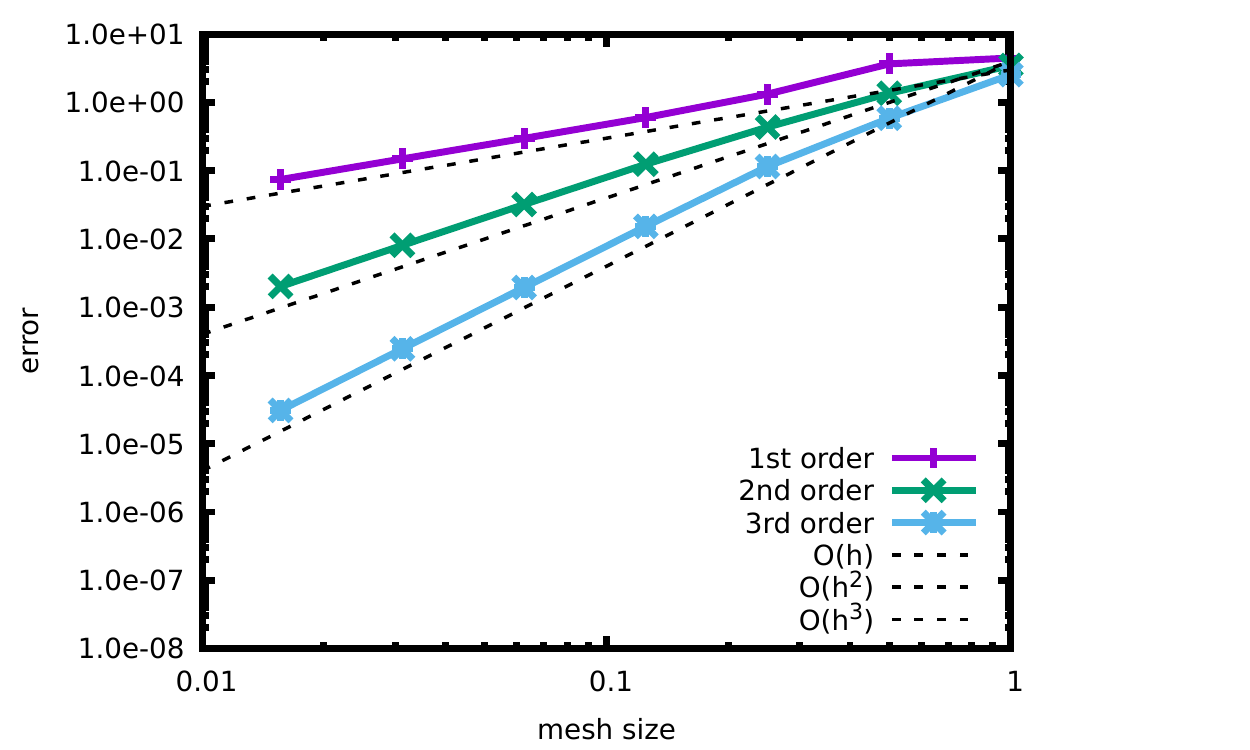}
   \label{fig:inv_ster_proj_approximation_error}
  }

  \subfigure[Geodesic finite elements]{
   \includegraphics[width=0.49\textwidth]{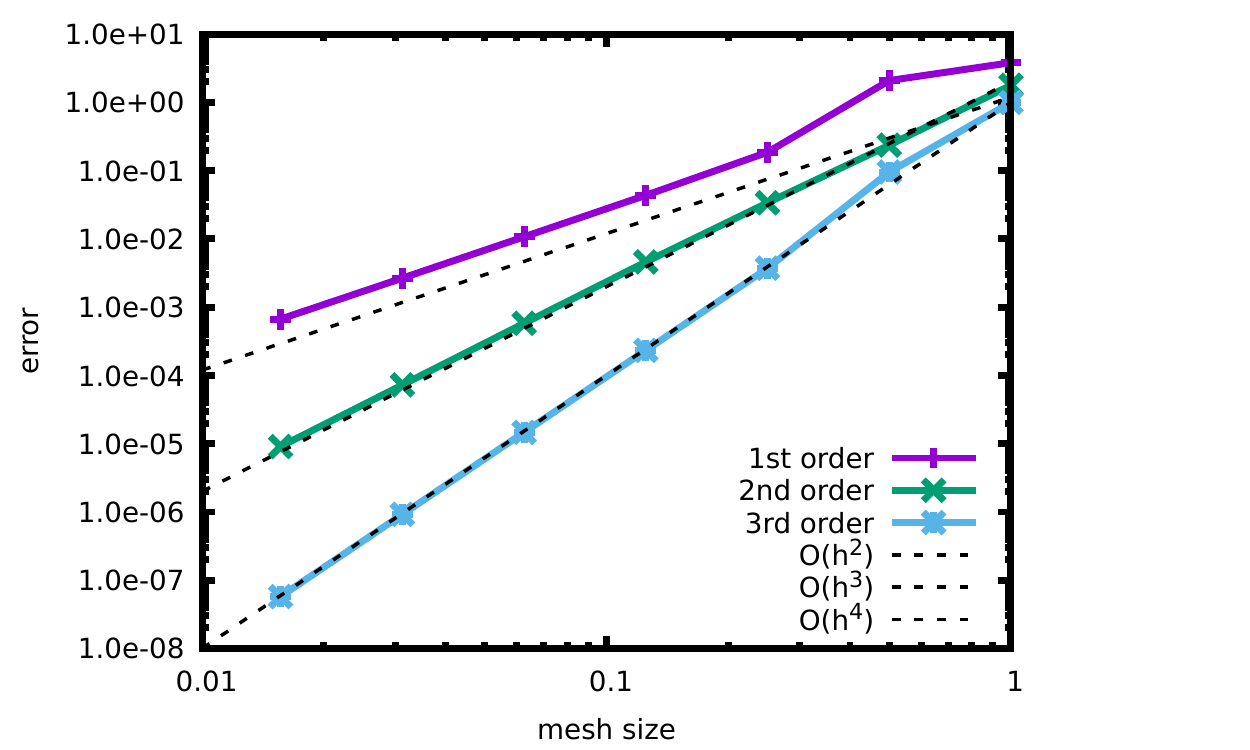}
   \includegraphics[width=0.49\textwidth]{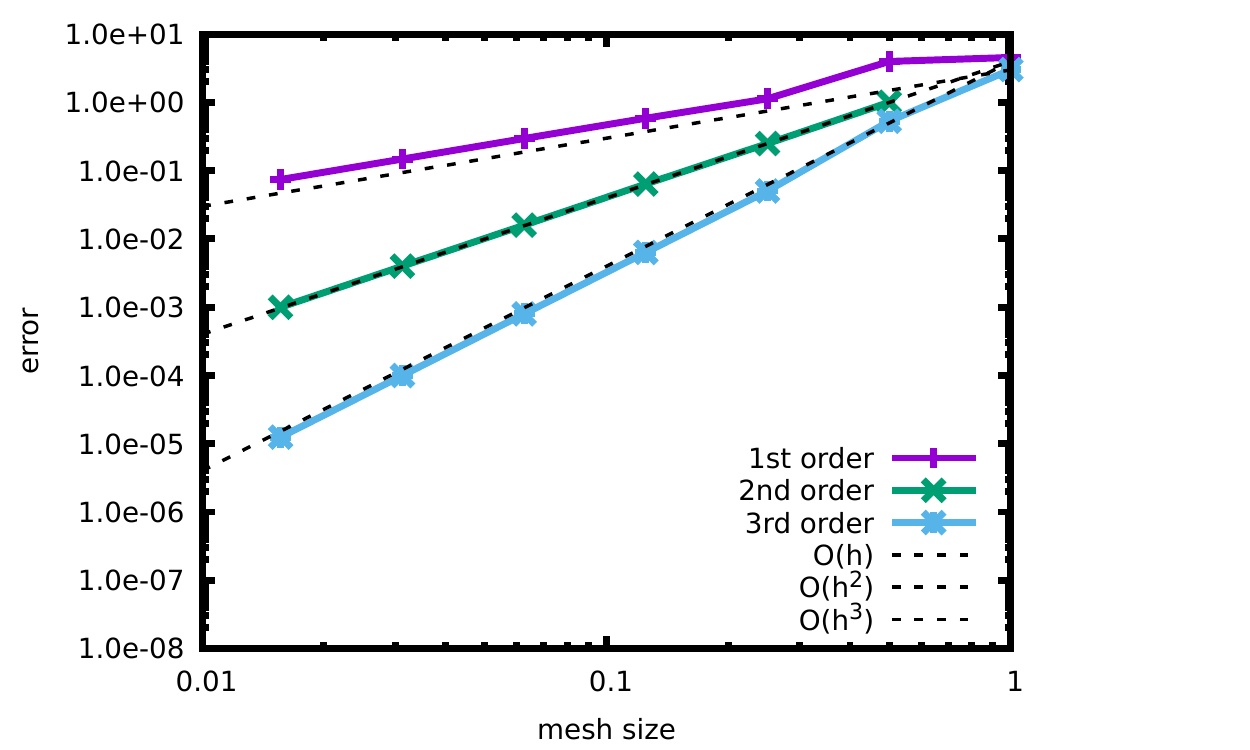}
   \label{fig:inv_ster_proj_approximation_error_gfe}
  }
 \end{center}
 \caption{Interpolation errors for the inverse stereographic projection.
          Left: $L^2$ norm, right: $H^1$ seminorm.
          The black dashed lines are at the same positions for both discretizations.}

\end{figure}

The results are plotted in Figures~\ref{fig:inv_ster_proj_approximation_error}.
As expected, the errors decay like $h^{p+1}$ for the $L^2$-norm, and like $h^p$ for
the $H^1$-seminorm.  These are the optimal orders predicted by the theory in Chapter~\ref{sec::errest}.

We now compare the projection-based discretization to a discretization using geodesic finite elements.
The resulting errors per mesh size are shown in Figure~\ref{fig:inv_ster_proj_approximation_error_gfe}.
One can see that the same asymptotic orders are obtained, as predicted by the interpolation
theory for geodesic finite elements~\cite{grohs_hardering_sander:2015,hardering:2016}.
However, the constant is consistently better for geodesic finite elements for orders $p=2$ and $p=3$.
Supposedly, the reason for this is that the intrinsic construction of geodesic finite elements captures
the geometry of $S^2$ better.

\subsection{Maps into the special orthogonal group}
\label{sec:approximation_error_so3}

\begin{figure}
 \begin{center}
  \subfigure[Projection-based finite elements]{
   \includegraphics[width=0.49\textwidth]{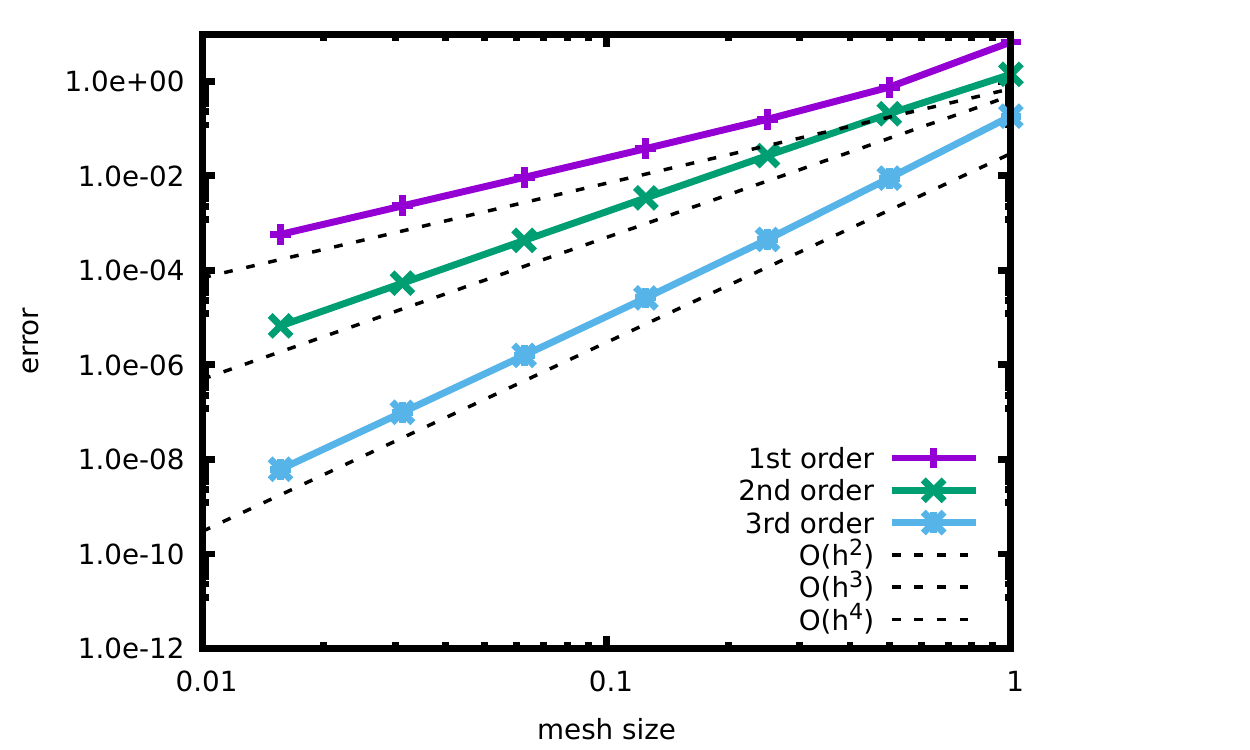}
   \includegraphics[width=0.49\textwidth]{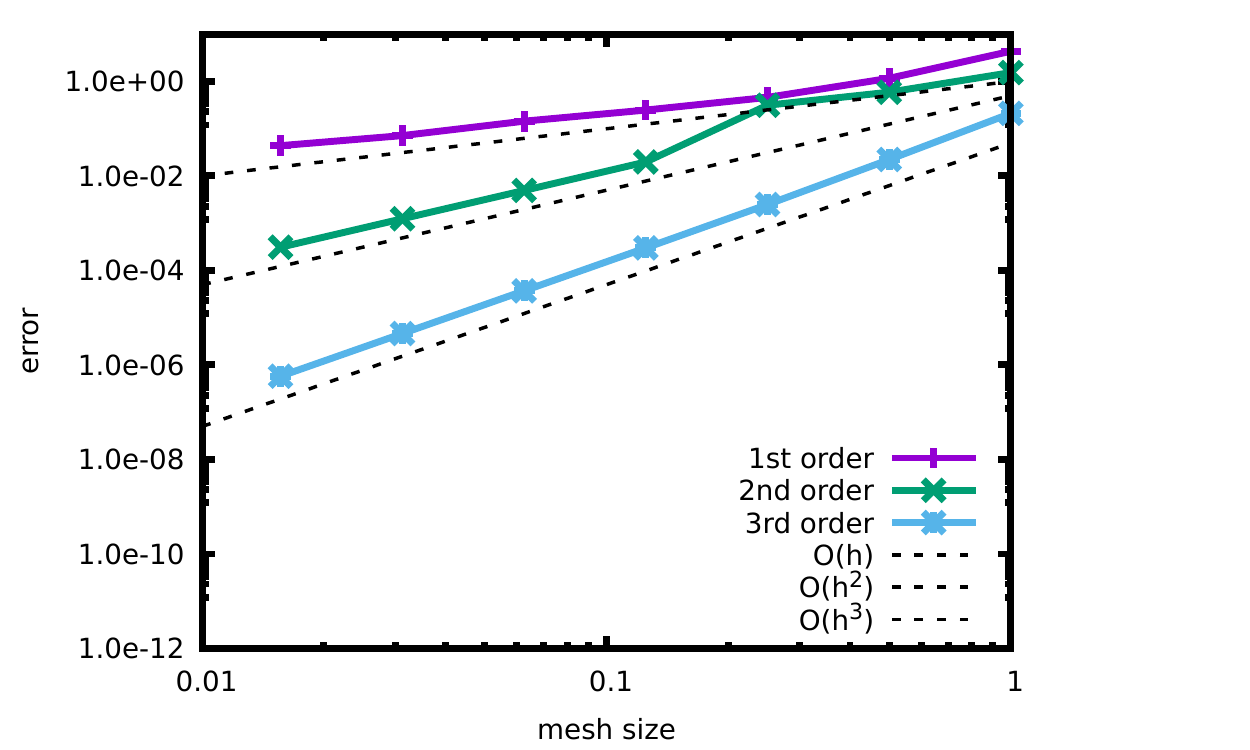}
   \label{fig:so3_approximation_error}
  }

  \subfigure[Geodesic finite elements]{
   \includegraphics[width=0.49\textwidth]{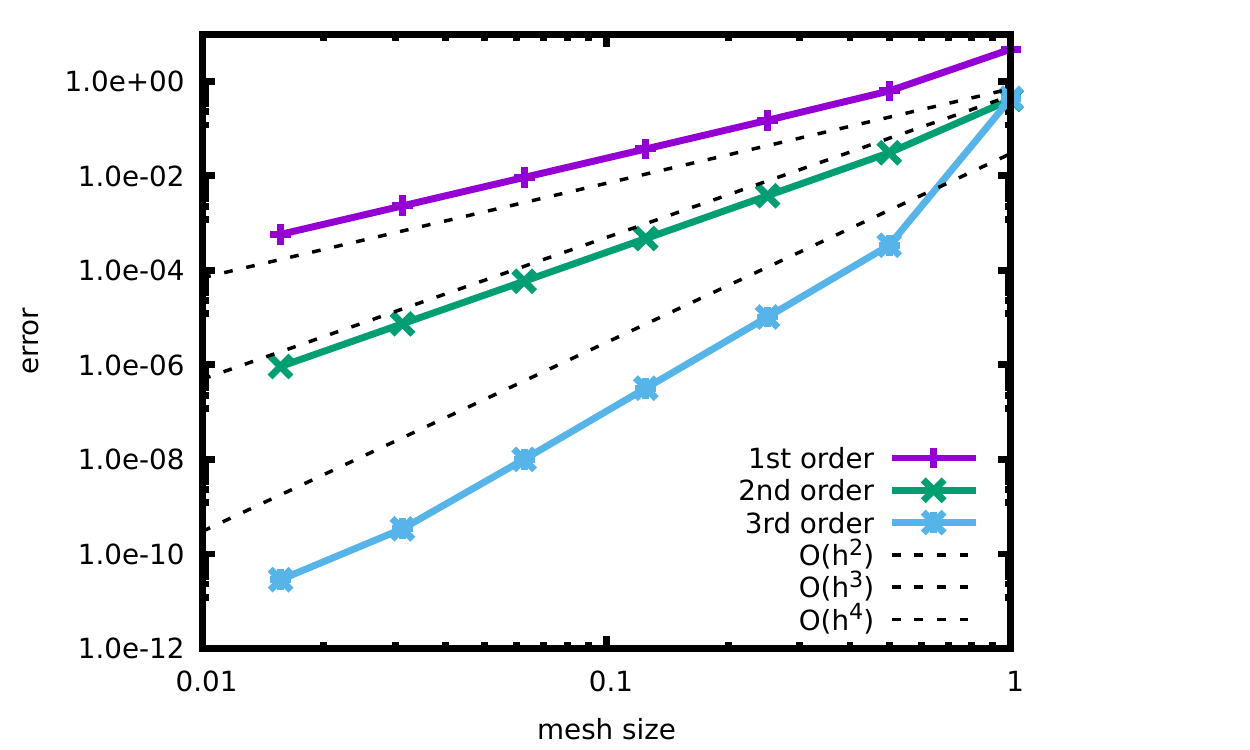}
   \includegraphics[width=0.49\textwidth]{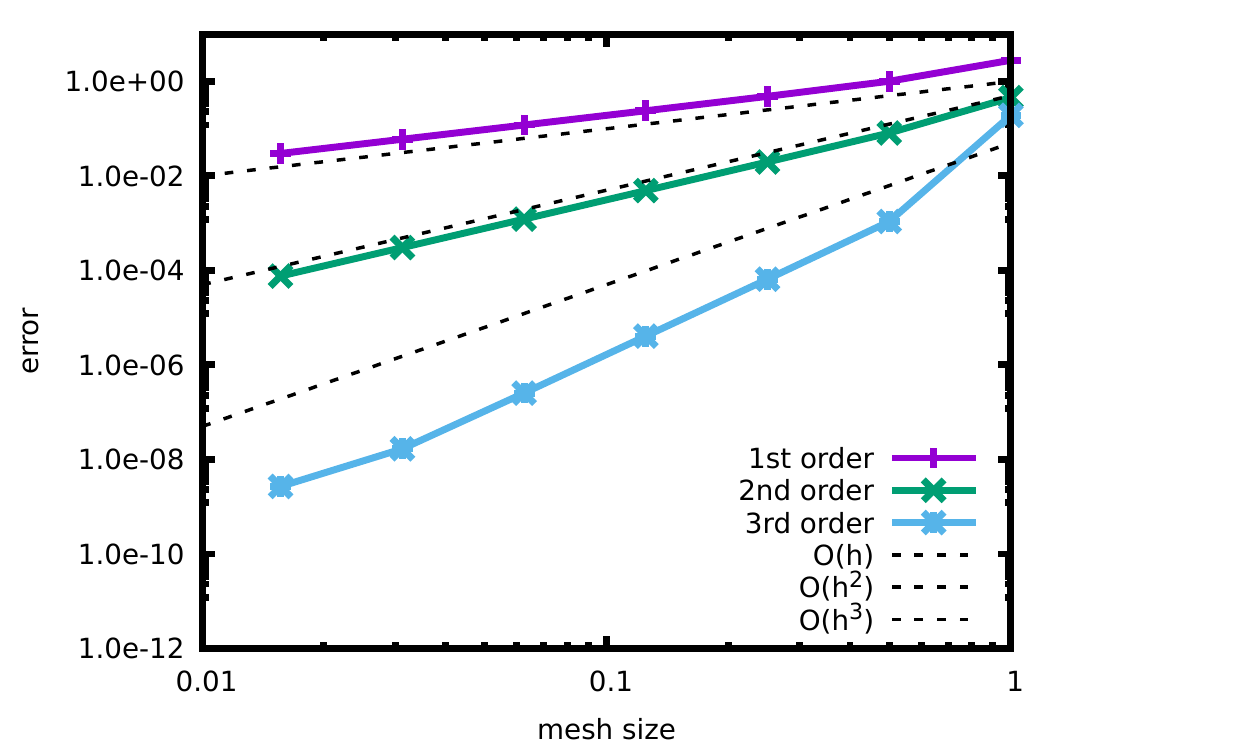}
   \label{fig:so3_approximation_error_gfe}
  }
 \end{center}
 \caption{Interpolation errors for a map into $\text{SO}(3)$.
          Left: $L^2$~norm, right: $H^1$ seminorm.
          The black dashed lines are at the same positions for both discretizations.}

\end{figure}

We repeat the experiment for a map into $\text{SO}(3)$.
As shown in~\cite{neff_lankeit_madeo:2014}, the closest-point projection $P$ of a general matrix
onto $\text{SO}(3)$ is
the orthogonal factor of the polar decomposition.  As such, $P(A)$ is defined for all $A \in \R^{3 \times 3}$,
and it is unique if $A$ is invertible.  In particular, this will be the case for all $A$ close enough to $\text{SO}(3)$.

To numerically compute the polar factor $P(A)$ of $A$ we use the iteration defined by $Q_0\colonequals A$ and
\begin{equation}
\label{eq:highams_method}
Q_{k+1}\colonequals \frac{1}{2}\left(Q_k+Q_k^{-T} \right),
\end{equation}
which is based on Heron's method for computing the square root of $1$.
\citet{higham} showed that this iteration converges quadratically to $P(A)$.
To compute $H^1$ norms of projection-based finite elements
we also need the derivative of the polar factor with respect to $A$.
Following \cite{gawlik_leok:2017b}, we use the iterative algorithm that results from differentiating~\eqref{eq:highams_method}.

For the domain of the example we use $\Omega = (-5,5)^2$ of the previous section,
and we also reuse the grid from Figure~\ref{fig:inverse_stereographic_projection}.
We will interpolate the function $R : \Omega \to \text{SO}(3)$,
\begin{equation}
\label{eq:SO3_test_function}
 R(x)
 \colonequals
 \begin{pmatrix}
  1 & 0 & 0 \\
  0 & \cos (\frac{\pi}{5} x_0) & -\sin (\frac{\pi}{5} x_0) \\
  0 & \sin (\frac{\pi}{5} x_0) &  \cos (\frac{\pi}{5} x_0)
 \end{pmatrix}
 \begin{pmatrix}
  \cos (\frac{\pi}{5} x_1) & 0 & -\sin (\frac{\pi}{5} x_1) \\
   0       & 1 & 0         \\
  \sin (\frac{\pi}{5} x_1) & 0 &  \cos (\frac{\pi}{5} x_1)
 \end{pmatrix}.
\end{equation}
Again we measure the errors in the $L^2$ norm and the $H^1$ seminorm on a set of seven grids
obtained by uniform refinement.
Interpolation errors for this scenario are plotted in Figure~\ref{fig:so3_approximation_error}.
As expected, we see the same optimal orders as for the case of mapping into the sphere.
Figure~\ref{fig:so3_approximation_error_gfe} shows the corresponding errors obtained using
geodesic finite elements. One observes the same effect as before: the convergence rates
are the same, but the constant is lower.
In fact, this effect is now more pronounced than before, and seems to increase with the order $p$.

\section{Numerical discretization error tests for harmonic maps}

In this second chapter of numerical results we present measurements of the discretization errors
of harmonic maps into $S^2$ and $\text{SO}(3)$. These confirm the theoretical predictions
of Chapter~\ref{disc_err_bnd}.  In addition,
we again compare the results to geodesic finite elements.

\subsection{Harmonic maps into the sphere \texorpdfstring{$S^2$}{S2}}

\begin{figure}
 \begin{center}
  \subfigure[Projection-based finite elements]{
   \includegraphics[width=0.46\textwidth]{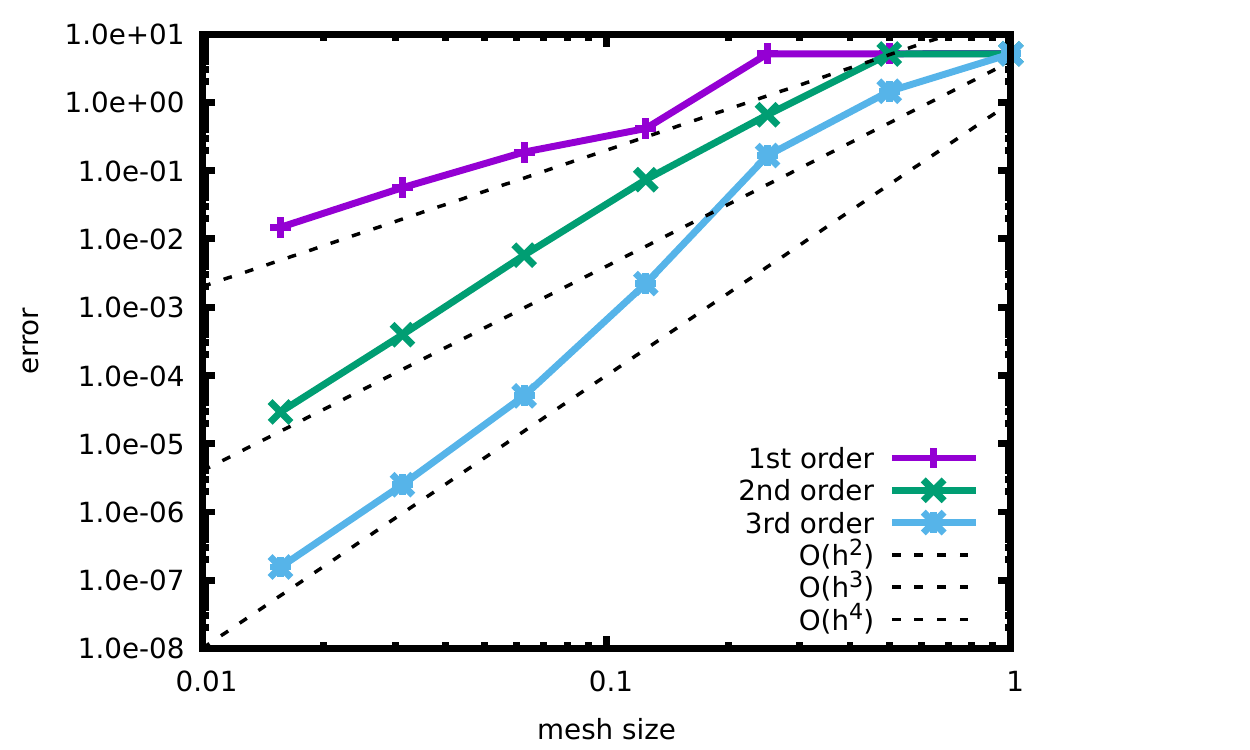}
   \includegraphics[width=0.46\textwidth]{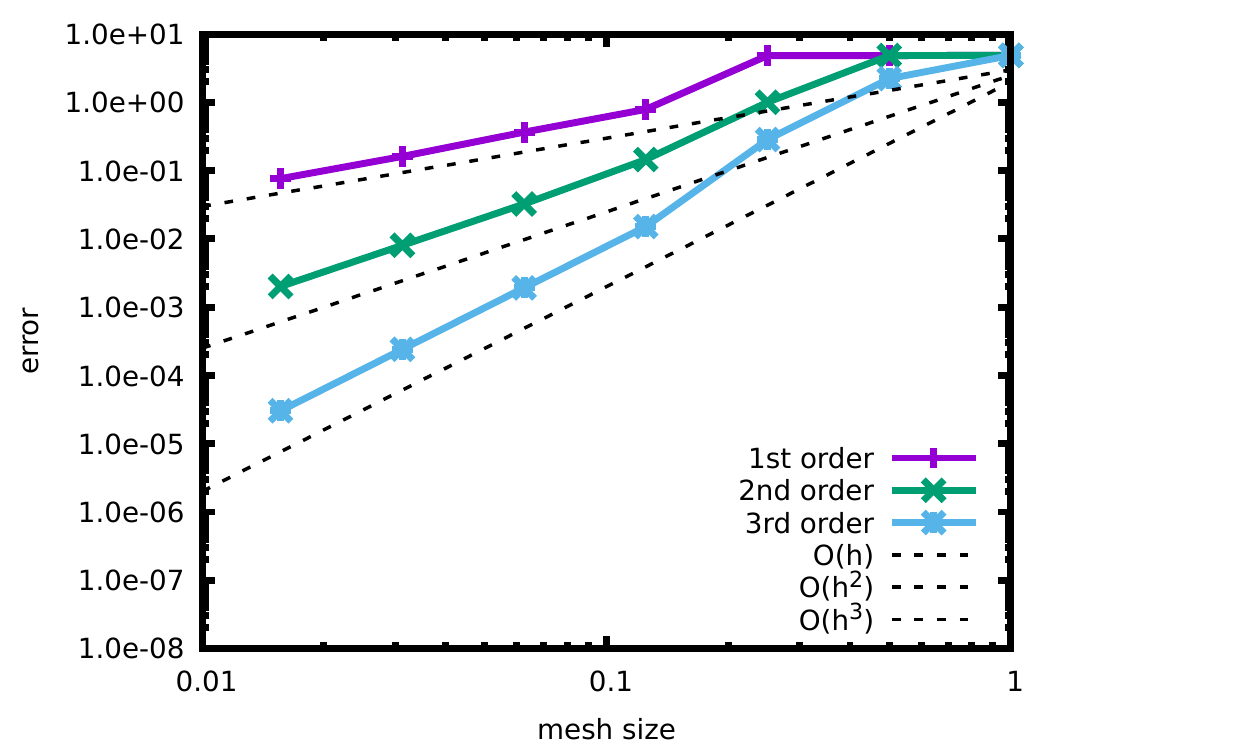}
   \label{fig:discretization_errors}
  }

  \subfigure[Geodesic finite elements]{
   \includegraphics[width=0.46\textwidth]{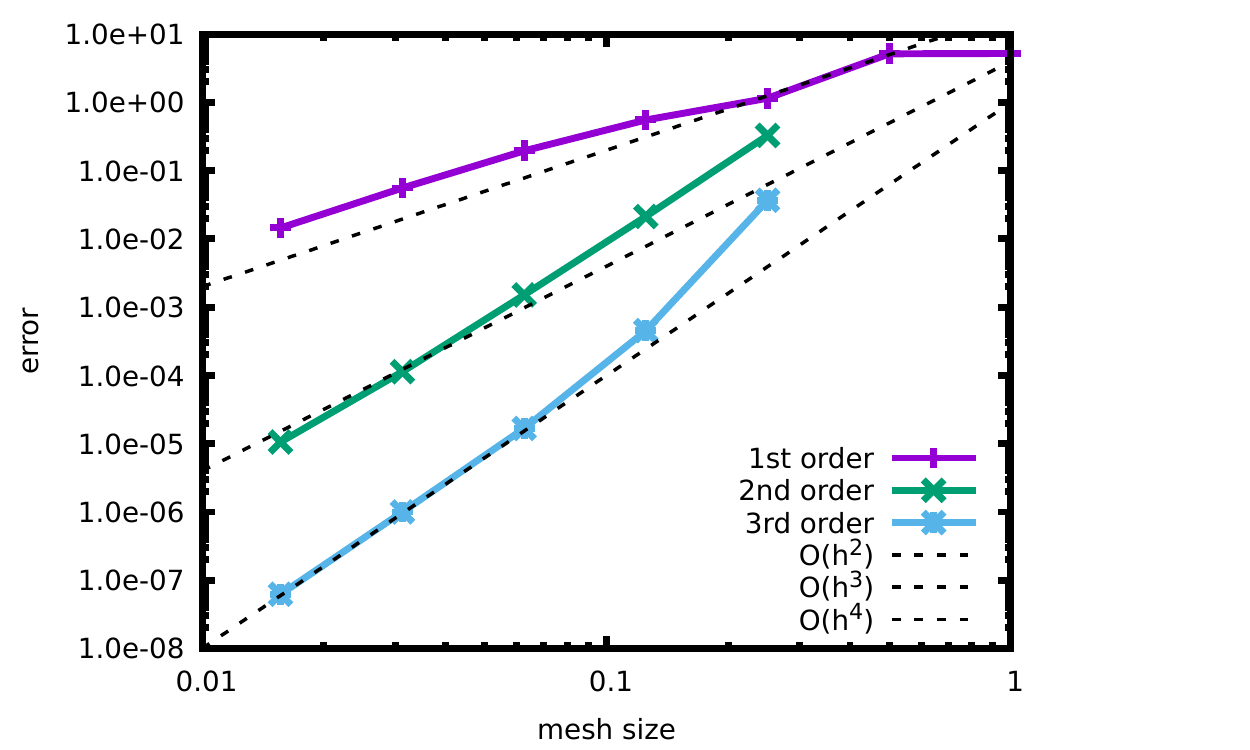}
   \includegraphics[width=0.46\textwidth]{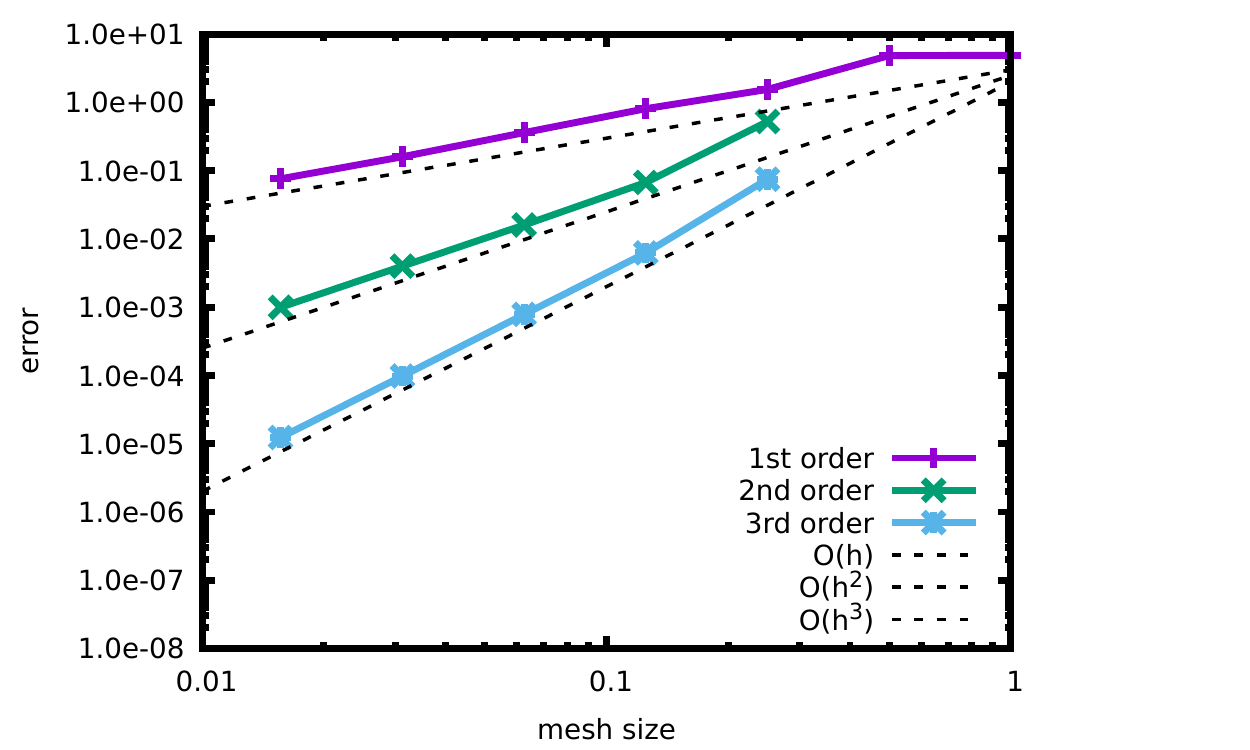}
   \label{fig:discretization_errors-gfe}
  }
\end{center}
 \caption{Discretization errors  for a harmonic map into $S^2$ as a function of normalized grid edge length.
 Left: $L^2$-norm. Right: $H^1$-seminorm.
 The black dashed lines are at the same positions for both discretizations.}
\end{figure}

The example builds on top of the interpolation error measurements of Section~\ref{sec:interpolation_errors_sphere}.
Reusing the domain $\Omega = (-5,5)^2$ and grids from there, we compute harmonic maps
from $\Omega$ into $S^2$ that take the values given by the inverse stereographic projection
function $\invstproj$ defined in~\eqref{eq:inverse_stereographic_projection} on the boundary of $\Omega$
in spaces of projection-based finite elements of orders $p=1,2,3$ mapping into $S^2$.
It is shown by \cite{belavin_polyakov:1975} (see also \cite{melcher:2014}), that the inverse
stereographic projection $\invstproj$ itself
minimizes the harmonic energy~\eqref{eq:harmEnergy} in the first non-trivial homotopy group in $H^1(\R^2, S^2)$.
This function is $C^\infty$ and we can therefore hope for optimal discretization error orders.

Using the canonical embedding of $S^2$ into $\R^3$, and the metric on $S^2$ induced by
the embedding, the integrand $\abs{\nabla v}^2$ of~\eqref{eq:harmEnergy} has the coordinate representation
\begin{equation*}
 \abs{\nabla v}^2 =  \sum_{i=1}^s \sum_{a=1}^3 \Big(\frac{\partial v^a}{\partial x^i} \Big)^2,
\end{equation*}
that is, $\nabla v$ is a $3 \times s$-matrix and $\abs{\,\cdot\,}$ the Frobenius norm.
We compute minimizers of the discrete energy using the approach proposed in~\cite{sander2012}
for geodesic finite elements. Identifying discrete functions with sets of coefficients in $S^2$,
we obtain an algebraic minimization problem posed on the nonlinear manifold
$(S^2)^n$, where $n$ is the number of Lagrange nodes on the grid.
This minimization problem is solved using
the Riemannian trust-region method introduced by \cite{absil_mahony_sepulchre:2008}
together with the inner monotone multigrid solver described in \cite{sander2012}.
Gradient and Hessian of the energy functional are computed using the ADOL-C
automatic differentiation software \cite{walther_griewank:2012}, and
the formula derived in~\cite{absil_mahony_trumpf:2013} to obtain the Riemannian Hessian
from the Euclidean Hessian.

The Riemannian trust-region solver is set to iterate until the maximum norm of the correction
drops below $10^{-6}$.  We then compute solutions $(\invstproj)_p^k$, $k=0,\dots,6$, $p=1,2,3$
on the grids obtained by $k$ steps of uniform refinement, and compute the errors
\begin{equation*}
 e^k_p = \norm{v^k_p - \invstproj},
\qquad
k=0,\dots,6,
\qquad
p=1,2,3,
\end{equation*}
where $\norm{\cdot}$ is either the norm in $L^2(\Omega,\R^3)$, or the half norm in $H^1(\Omega,\R^3)$.
Figure~\ref{fig:discretization_errors} shows the errors $e_p^k$ as functions of the normalized mesh
size $h$. We see that for $p$-th order finite elements the $L^2$-error decreases
like $h^{p+1}$, and the $H^1$-error decreases like $h^p$.
Hence we can reproduce the optimal convergence behavior predicted by
Theorems~\ref{T:H1Discerr} and~\ref{thm:L2_discretization_error_bound}.

We now compare the projection-based discretization to a discretization using geodesic finite elements.
The resulting errors per mesh size are shown in Figure~\ref{fig:discretization_errors-gfe}.
One can see that the same asymptotic orders are obtained, as predicted by
theory~\cite{grohs_hardering_sander:2015,hardering:2018}. As in the interpolation case,
the constant is slightly better for geodesic finite elements.
On the other hand, one can see that the graphs in Figure~\ref{fig:discretization_errors-gfe} do not
contain values for the two coarsest grids and approximation orders~2 and~3. This is because the
minimization problem that defines geodesic interpolation was actually ill-defined on at least
one grid element in these cases. The problem does not happen for projection-based finite elements
for this example.

\begin{figure}
 \begin{center}
  \includegraphics[width=0.5\textwidth]{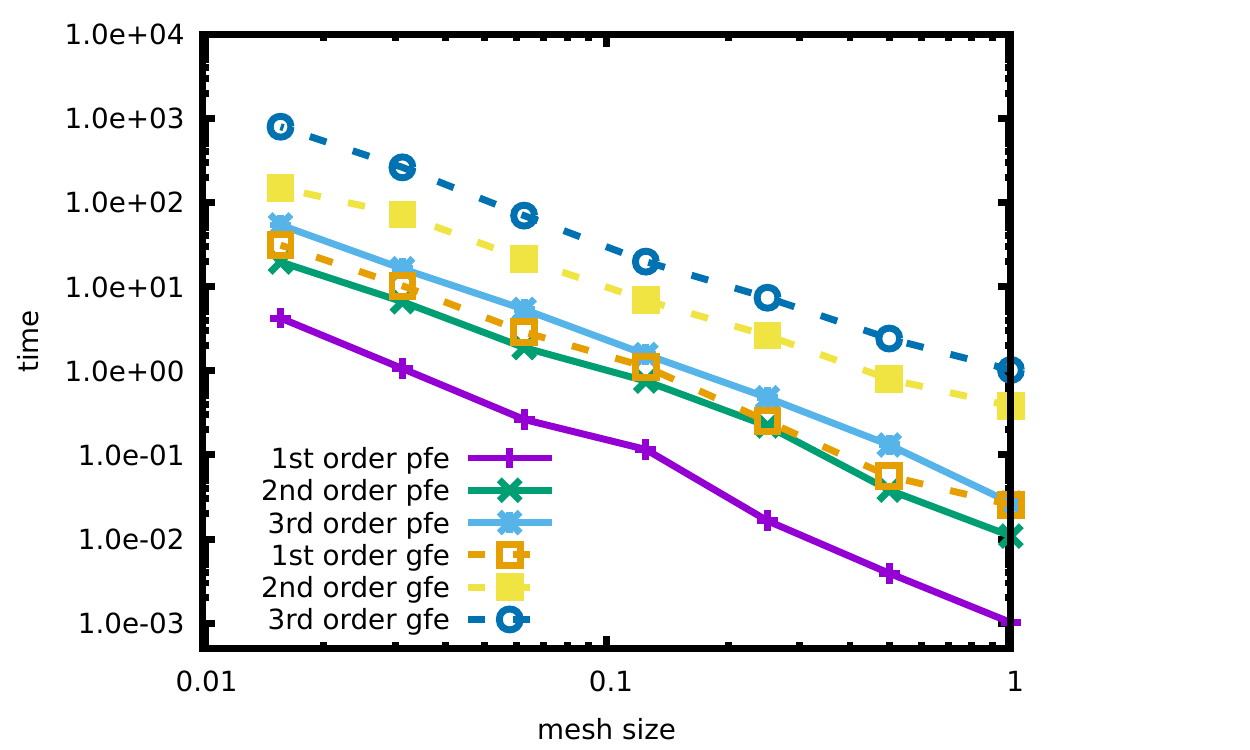}
 \end{center}
 \caption{Wall-time needed to compute the harmonic energy on seven different grids and
   approximation orders $p = 1, 2, 3$. Solid: projection-based finite elements.
   Dashed: geodesic finite elements.}
 \label{fig:run_time_stereographic}
\end{figure}

The decisive argument for projection-based finite elements for this scenario, however, is run-time.
Figure~\ref{fig:run_time_stereographic} plots the total time needed to compute the harmonic energy
for the different finite element spaces and grid resolutions.
Projection-based finite elements need only about 10\,\% of the time of geodesic finite elements.
This is of course because projection-based interpolation is given by a simple closed-form formula
in the case of $M=S^2$, whereas for geodesic finite elements it involves numerically solving
a small minimization problem~\eqref{eq:geodesic_interpolation}.
By means of automatic differentiation, these differences appear
in the computation of derivatives as well.
In practical applications of sphere-valued problems, projection-based finite elements are therefore
typically preferable to geodesic finite elements.

\subsection{Harmonic maps into \texorpdfstring{$\text{SO}(3)$}{SO(3)}}

In the final example we compute minimizers of the harmonic energy in a space of functions
mapping to $\text{SO}(3)$. As in Section~\ref{sec:approximation_error_so3}, we use the canonical
embedding of $\text{SO(3)}$ into $\R^{3 \times 3}$, and the polar factor as the projection
onto $\text{SO}(3)$ (even though the implementation uses quaternions to actually store
elements of $\text{SO}(3)$). The iteration~\eqref{eq:highams_method}
used to compute the polar factor is a variant of a Newton method, and therefore
plays nicely with automatic differentiation systems like ADOL-C~\cite[Chap.\,15]{griewank_walther:2008}.

We base our numerical test on the interpolation error test of Section~\ref{sec:approximation_error_so3}.
On the domain $\Omega$ given there, we look for minimizers of the harmonic energy in $H^1(\Omega,\text{SO}(3))$,
subject to Dirichlet boundary and homotopy constraints given by the function $R$ defined in \eqref{eq:SO3_test_function}.
As the solution of this problem is not known in closed form, we compute discretization errors with respect
to a numerical reference solution. For this we refine the grid uniformly 6~times, and compute the solution
there. We then trust this to be a good reference solution for grids with up to 5~steps of refinement.

\begin{figure}
 \begin{center}
  \subfigure[Projection-based finite elements]{
   \includegraphics[width=0.46\textwidth]{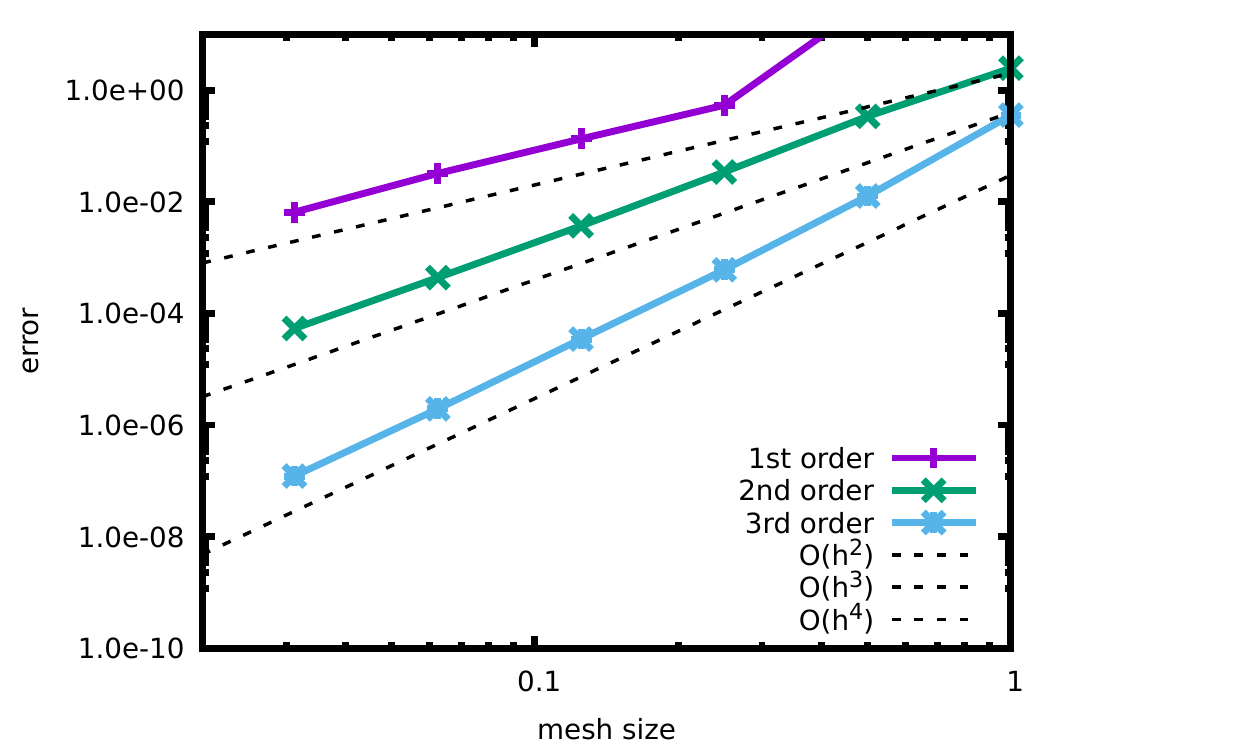}
   \includegraphics[width=0.46\textwidth]{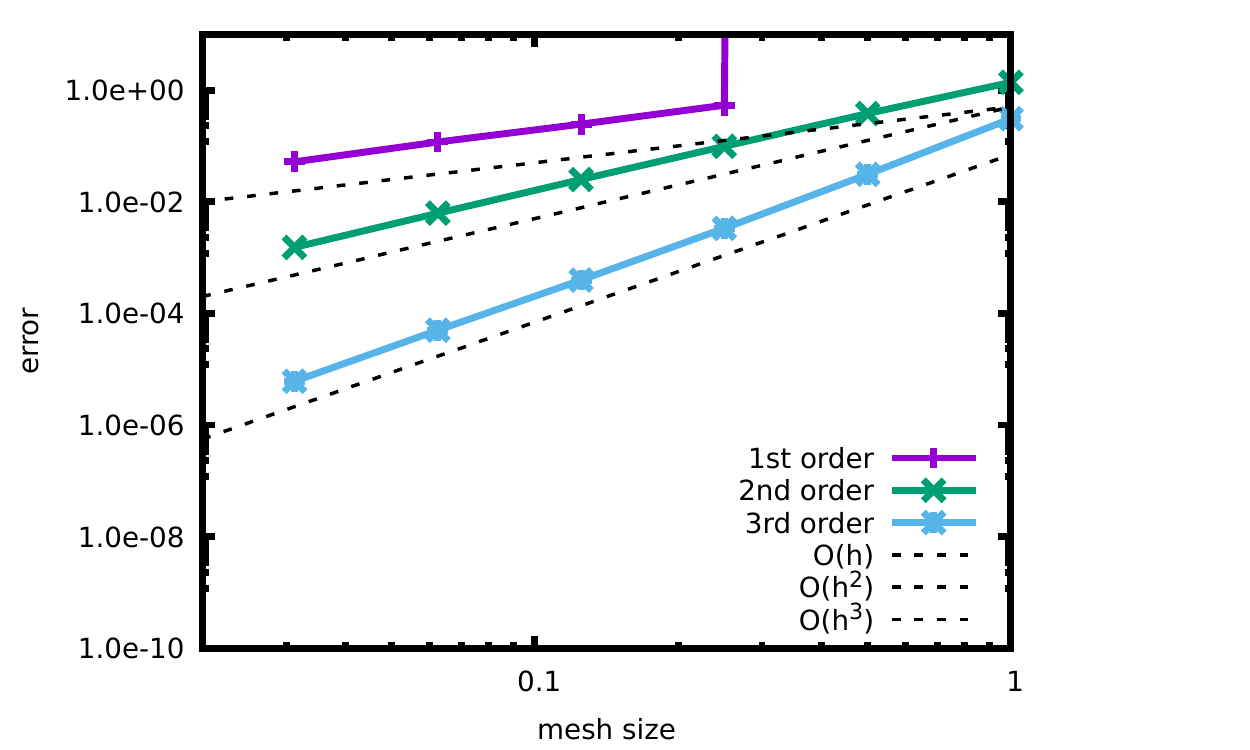}
   \label{fig:discretization_errors_so3}
  }

  \subfigure[Geodesic finite elements]{
   \includegraphics[width=0.46\textwidth]{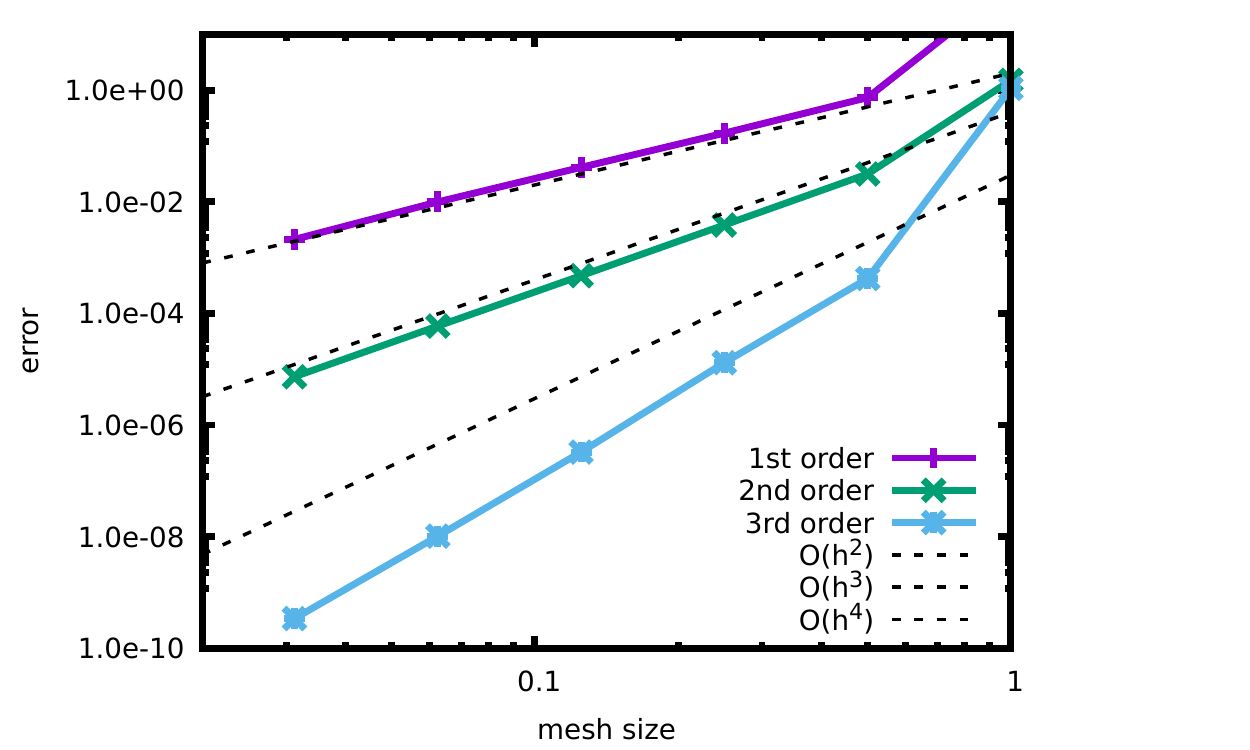}
   \includegraphics[width=0.46\textwidth]{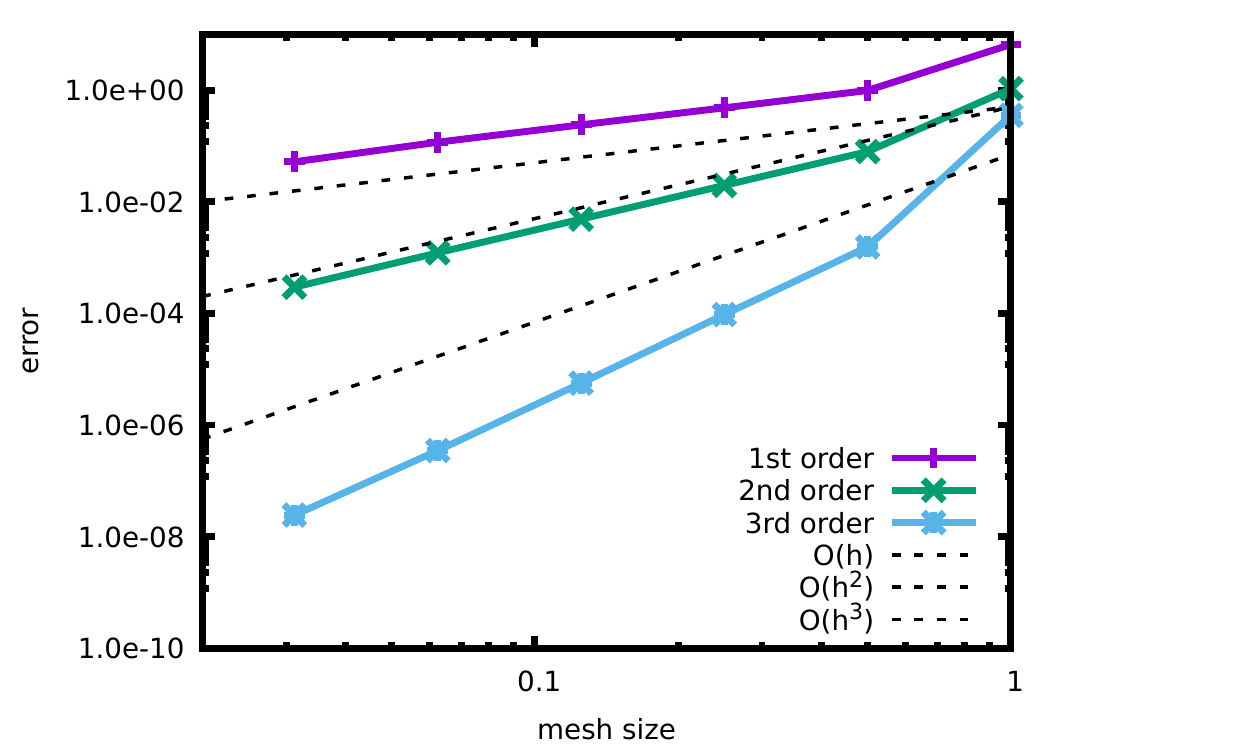}
   \label{fig:discretization_errors_so3_gfe}
  }
 \end{center}
 \caption{Discretization errors for a harmonic map into $\text{SO}(3)$ as a function of normalized grid edge length.
 Left: $L^2$~norm. Right: $H^1$~seminorm.
 The black dashed lines are at the same positions for both discretizations.}
\end{figure}

Figure~\ref{fig:discretization_errors_so3} shows the discretization error plots for the $L^2$ and $H^1$
errors, again for approximation spaces of orders up to $3$. We observe the expected optimal discretization
error rates in all cases.

Finally, we redo the experiment with geodesic finite elements.  Figure~\ref{fig:discretization_errors_so3_gfe}
shows the discretization errors per mesh size for the same problem, but using a geodesic finite element
discretization. The constant is again better than for projection-based finite elements,
and the difference seems to increase with $p$.

\begin{figure}
 \begin{center}
  \includegraphics[width=0.5\textwidth]{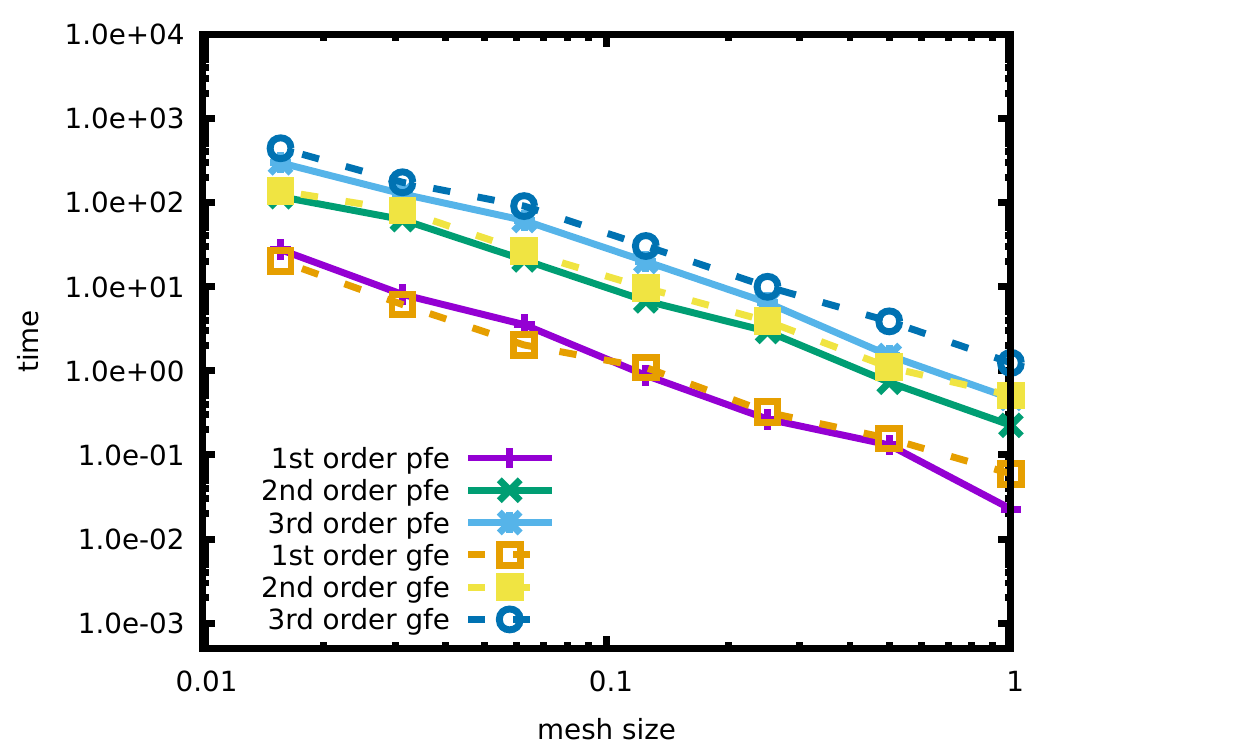}
 \end{center}
 \caption{Wall-time needed to compute the harmonic energy of the function $R$ on seven different grids and
   approximation orders $p = 1, 2, 3$. Solid: projection-based finite elements.
   Dashed: geodesic finite elements.}
 \label{fig:run_time_so3_synthetic}
\end{figure}

When comparing the run-times again (Figure~\ref{fig:run_time_so3_synthetic}), the situation is vastly different.
While projection-based finite elements were much faster than geodesic ones for the case of maps
into the sphere, there is hardly a difference for $M = \text{SO}(3)$. This is because the
projection from $\R^{3\times 3}$ to $\text{SO}(3)$ is not given in a simple closed form, but has
to be computed iteratively (Section~\ref{sec:approximation_error_so3}).
This puts the execution speed of projection-based finite elements on par with geodesic finite elements,
if the projection is the polar decomposition in $\R^{3 \times 3}$.

%\printbibliography
\bibliographystyle{abbrvnat}
\bibliography{grohs-hardering-sander-sprecher-pfe-errors}

%------------------------------

\end{document}